\documentclass[12pt]{amsart}
\usepackage{latexsym,amscd,amssymb,epsfig} 
\usepackage[dvips]{color}
\usepackage[all]{xy}
\usepackage[T1]{CJKutf8}

\theoremstyle{plain}
  \newtheorem{theorem}{Theorem}[section]
  \newtheorem{proposition}[theorem]{Proposition}
  \newtheorem{lemma}[theorem]{Lemma}
  \newtheorem{corollary}[theorem]{Corollary}
  
\theoremstyle{definition}
  \newtheorem{definition}[theorem]{Definition}
  \newtheorem{example}[theorem]{Example}

\theoremstyle{remark}
  \newtheorem{remark}[theorem]{Remark}

\numberwithin{equation}{section}

\def\umapright#1{\smash{
   \mathop{\longrightarrow}\limits^{#1}}}

\def\rmapdown#1{\Big\downarrow\rlap
   {$\vcenter{\hbox{$\scriptstyle#1$}}$}}

\def\tempbaselines
{\baselineskip22pt\lineskip3pt
   \lineskiplimit3pt}
\def\diagram#1{\null\,\vcenter{\tempbaselines
\mathsurround=0pt
    \ialign{\hfil$##$\hfil&&\quad\hfil$##$\hfil\crcr
      \mathstrut\crcr\noalign{\kern-\baselineskip}
  #1\crcr\mathstrut\crcr\noalign{\kern-\baselineskip}}}\,}

\def\pullback#1&#2&#3&#4&#5&#6&#7&#8&{
\diagram{#1&\umapright{#2}&#3\cr
\rmapdown{#4}&&\rmapdown{#5}\cr
#6&\umapright{#7}&#8\cr}}

\def\calC{{\mathcal C}}
\def\calD{{\mathcal D}}
\def\calE{{\mathcal E}}

\def\calO{{\mathcal O}}
\def\calP{{\mathcal P}}
\def\calQ{{\mathcal Q}}
\def\calR{{\mathcal R}}
\def\calS{{\mathcal S}}
\def\calV{{\mathcal V}}

\def\R{{\underline{R}}}
\def\k{{\underline{k}}}

\def \Aut{\mathop{\rm Aut}\nolimits}

\def\colim{\mathop{\varprojlim}\nolimits}

\def \H{\mathop{\rm H}\nolimits}

\def \Hom{\mathop{\rm Hom}\nolimits}

\def \Id{\mathop{\rm Id}\nolimits}

\def\lim{\mathop{\varinjlim}\nolimits}
 
\def \Ob{\mathop{\rm Ob}\nolimits} 
\def \Mor{\mathop{\rm Mor}\nolimits} 

\def\Res{\mathop{\rm Res}\nolimits}

\def\Stab{\mathop{\rm Stab}\nolimits}
\def\supp{\mathop{\rm supp}\nolimits}

\def\SG{{S\rtimes G}}
\def\SH{{S\rtimes H}}
\def\SK{{S\rtimes K}}
\def\PH{{\calP\rtimes H}}
\def\PG{{\calP\rtimes G}}

\def\QH{{\calQ\rtimes H}}

\def\RK{{\calR\rtimes K}}

\def\VD{{\calV\rtimes D}}

\def\CC{{\Bbb C}}

\def\ZZ{{\Bbb Z}}

\begin{document}

\title[Transporter category algebras]
{Local representation theory of transporter categories}


\author{Fei Xu}
\email{fxu@stu.edu.cn}
\address{Department of Mathematics\\
Shantou University\\
Shantou, Guangdong 515063, China}

\subjclass[2010]{}


\thanks{The author \begin{CJK*}{UTF8}{}
\CJKtilde \CJKfamily{gbsn}(徐 斐)
\end{CJK*} is supported by the NSFC grant No. 11671245}

\begin{abstract}
We attempt to generalize the $p$-modular representation theory of finite groups to finite transporter categories, which are regarded as generalized groups. We shall carry on our tasks through modules of transporter category algebras, a type of Gorenstein skew group algebras. The Kan extensions, upgrading the induction and co-induction, are our main tools to establish connections between representations of a transporter category and of its transporter subcategories. Some important constructions and theorems in local representation theory of finite groups are generalized.
\end{abstract}

\maketitle

\section{Introduction}

Let $G$ be a finite group and $\calP$ be a finite $G$-poset (we shall regard a $G$-set as a $G$-poset with trivial relations). The \textit{transporter category} over $\calP$ is a \textit{Grothendieck construction} $\PG$, a specifically designed finite category. It may be thought as a semi-direct product between $G$ and $\calP$, and is considered as a generalized group. This construction has its roots in group theory, representation theory and algebraic topology. Our initiative comes from the observation that there exists a category equivalence $(G/H)\rtimes G\simeq H$ for each subgroup $H\subset G$.

In our earlier work, we investigated homological properties of the \textit{category algebra} $R\PG$, where $R$ is a commutative ring with identity. It is known that the category of finitely generated left modules, $R\PG$-mod, is a symmetric monoidal category. Based on this, we studied the representation theory of $R\PG$, and its connections with representations of groups \cite{Xu1, Xu3, Xu4}. In this way, we generalized some well-known results in group representations and cohomology, and provided new insights into certain existing results.

In the present paper, we examine transporter categories (as generalized groups) from a different point of view. Our treatment allows some classical settings in local representation theory (of groups), and the results that follow, to survive in this generality. To this end, let $H$ be a subgroup of $G$ and $\calQ$ be a $H$-subposet of $\calP$. The category $\QH$ is called a \textit{transporter subcategory} of $\PG$. We will discuss the structure theory of transporter categories, based on which we shall develop a local representation theory. It means that we will establish connections between the representations of $\PG$ and those of its transporter subcategories.  The idea of using $\QH$ to understand $\PG$ may be traced back to the Quillen stratification of the equivariant cohomology ring $\H^*_G(B\calP,k)\cong\H^*(EG\times_GB\calP,k)$, in which $EG\times_GB\calP$ is indeed homotopy equivalent to the classifying space $B(\PG)$. Our ultimate aim is to investigate representations of various local categories, arisen in group representations and homotopy theory, and their applications, see for instance \cite{AKO}.

We shall carry on the above mentioned tasks with the help of transporter category algebras $k\PG$, where $k$ is an (algebraically closed) field of characteristic $p$ that divides the order of $G$. If $H$ happens to be a $p$-subgroup, we shall call $\QH$ a $p$-transporter subcategory. We have the following comparison chart. The bulk of this paper contains a theory of vertices and sources, as well as a theory of blocks, for transporter category algebras.

\begin{center}
\begin{tabular}{r|c|c}
& group representations & category representations\\
\hline
\hline
\scriptsize{structure} & $G$ & $\PG$\\
\hline
\scriptsize{substructure} & $H$ & $\QH$\\
\hline
\scriptsize{algebra} & $kG$ & $k\PG$\\
\hline
\scriptsize{canonical basis} & $G=\Mor G$ & $\Mor(\PG)$\\
\hline
\hline
\scriptsize{modules} & $kG$-mod, $kH$-mod & $k\PG$-mod, $k\QH$-mod\\
\hline
\scriptsize{trivial module} & $k$ & $\k$\\
\hline
\scriptsize{restriction} & $\downarrow^G_H$ & $\downarrow^{\PG}_{\QH}$\\
\hline
\scriptsize{left Kan extension} & $\uparrow^G_H$ & $\uparrow^{\PG}_{\QH}$\\
\hline
\scriptsize{right Kan extension} & $\Uparrow^G_H (\cong \uparrow^G_H)$ & $\Uparrow^{\PG}_{\QH} (\not\cong\uparrow^{\PG}_{\QH})$\\
\hline
\hline
\scriptsize{enveloping category}& $G^e\cong G\times G$ & $(\PG)^e\cong\calP^e\rtimes G^e$\\
\hline
\scriptsize{diagonal category} & $\delta(G)\cong G$ & $F(\calP)\rtimes\delta(G)\cong F(\calP)\rtimes G$\\
\hline
\scriptsize{block theory} & $kG\cong k\uparrow^{G^e}_{\delta(G)}$ & $k\PG\cong\k\uparrow^{(\PG)^e}_{F(\calP)\rtimes\delta(G)}$\\
\hline
\scriptsize{defect} & $p$-\scriptsize{group} $D\subset G$ & $p$-\scriptsize{category} $\calV\rtimes D\subset F(\calP)\rtimes G$\\
\hline
\end{tabular}
\end{center}

To see a concrete example, we may choose $\calP=\calS_p$, the poset of non-trivial $p$-subgroups of $G$, with conjugation action. Then $\calS_p\rtimes G$ is the usual $p$-transporter category ${\rm Tr}_p(G)$, containing all $p$-local subgroups of $G$, as automorphism groups of objects. By definition, a representation of $\calS_p\rtimes G$ is a covariant functor from $\calS_p\rtimes G$ to $Vect_k$, the category of finite-dimensional $k$-vector spaces. It can be thought as a diagram of representations of local subgroups $N_G(P)$, for a collection of $P\in\Ob\calS_p$. As an application of local categories, one may find a new way to reformulate the Alvis-Curtis duality when $G$ is a Chevalley group in \cite{Xu4}.

The main results, whose proofs depend on the explicit calculation of $\uparrow^{\PG}_{\QH}$, include
\begin{enumerate}
\item a generalized theory of vertices and sources, including a Mackey formula (Theorem 4.18) and a Green correspondence (Theorem 4.26),

\item a comparison between the theories for $kG$-modules and of constance $k\PG$-modules (Proposition 4.23 and Corollary 4.27), and

\item a generalized Brauer correspondence (Theorem 5.8).
\end{enumerate}

To set our work into the historical context, we note that the transporter category algebras are skew group algebras, and thus are fully group-graded algebras. This work is partially motivated by the papers on fully group-graded algebras by Boisen \cite{Bo1}, Dade \cite{Da1, Da2, Da3}, and Miyashita \cite{Mi} (the latter in the context of $G$-Galois theory). Especially, Dade conceived a theory of vertices and sources (for fully group-graded algebras). However, his ``vertices'' seem to be too big, see Examples 4.15 and 4.24. Same problem occurs in Boisen's definition of a ``defect'' of a block, because a ``defect'' is a ``vertex'', in the sense of Dade, of some module. We shall propose a sharpened definition of a vertex, incorporating our earlier work on general EI category algebras \cite{Xu1}, and prove it is appropriate. For the reader's convenience, some key constructions and results, from the previously mentioned papers, are quoted here. The approach in this paper is mostly parallel to the standard one for group representations. However the extra $G$-poset structure does require more than mere technicality. 

The paper is organized as follows. In Section 2, we recall relevant results for fully group-graded algebras. Then we examine local structures of transporter categories in Section 3. Subsequently the Kan extensions for investigating representations will be thoroughly discussed from the beginning of Section 4. A generalized theory of vertices and sources will be given. Finally in Section 5, we study the block theory of transporter category algebras.\\

\noindent{\sc Acknowledgement} The author would like to thank Peter Webb for stimulating discussions during his visits to Shantou in 2013 and 2014.

\section{Results from fully group-graded algebras}

In the present paper, we want to develop modular representation theory of transporter category algebras. Some known results on fully group-graded algebras of Boisen \cite{Bo1}, Dade \cite{Da1, Da2, Da3}, and Miyashita \cite{Mi}, will specialize to our situation and they will pave the way towards our key constructions. We shall quote these results mainly for skew group algebras. Some proof are given if they are needed in our presentation.

Let $R$ be a commutative ring with identity. Suppose $G$ is a group and $A$ is a $G$-graded $R$-ring. It means that, as $R$-modules, we have
$$
A=\bigoplus_{g\in G}A_g,
$$
satisfying $A_gA_h\subset A_{gh}$. If $A$ meets the extra condition that $A_gA_h= A_{gh}$, then we say $A$ is \textit{fully $G$-graded}. Suppose $H$ is a subgroup of $G$. We may define a subalgebra $A_H=\bigoplus_{h\in H}A_h$. Particularly $A_1$ becomes a subalgebra.

Suppose $S$ is an $R$-ring that admits a $G$-action. We say $S$ has a $G$-action, if there exists a group homomorphism $\phi : G\to\Aut(S)$. Under the circumstance, we also call $S$ a $G$-ring. We usually denote the $G$-action by ${}^gs=\phi(g)(s)$ for all $s\in S$ and $g\in G$. Then we may continue to define the \textit{skew group ring} $\SG$. As an $R$-module, it is simply $S\otimes_R RG$. For convenience, we write $\sum sg$, instead of $\sum s\otimes g$, for an element in the skew group ring. The multiplication is determined by $(sh)(tg)=s{}^hthg$, for $s,t\in S$ and $h,g\in G$. This ring contains subrings $\{s1\bigm| s\in S\}\cong S$ and $\{n1_Sg\bigm| n\in \ZZ,\; g\in G\}\cong (\ZZ/d\ZZ )G$  for some $d\in\ZZ$. We may wish to take a larger ring $R\subseteq Z(S)$  fixed by $G$ so that $RG\subseteq S\rtimes G$. We assume $S$ is free as an $R$-module. For the sake of simplicity, for each $s\in S$ and $g\in G$, we shall write $g=1_Sg$ and $s=s1$ as elements of $\SG$, when there is no confusion. 

The skew group ring $A=\SG$ is fully $G$-graded, if we put
$$
A_g=Sg=\{sg\bigm{|} s\in S\},
$$
for each $g\in G$. Here we shall mainly recall constructions and results by Dade \cite{Da1, Da2} and Boisen \cite{Bo1}. For future applications, we will only state known results from \cite{Da1, Da2, Bo1} in the special forms for skew group algebras.

We also note that Reiten and Riedtmann \cite{RR} studied the representation theory of skew group algebras over $R=\CC$, the complex numbers. See \cite{ARS} for another presentation.

If $H\subset G$, we have an inclusion $\SH\subset\SG$, and thus the induction
$$M\uparrow_{\SH}^{\SG}:= {}_\SG \SG\otimes_\SH M$$
and restriction
$$
M\downarrow_{\SH}^{\SG} := {}_\SH\SG\otimes_\SG M.
$$
For instance $\SG\cong S\uparrow_{S\rtimes 1}^{\SG}$. In \cite{Da1,Bo1}, these two functors are denoted by symbols $\uparrow^G_H$ and $\downarrow^G_H$ since $S$ is unchanged and it matches the special case of groups. We refrain from using the latter in order to be consistent throughout this paper.

In general, we have a decomposition
$$
M\uparrow^{\SG}_{\SH}=\bigoplus_{g_i\in [G/H]} g_i\otimes M,
$$
and the $\SG$-modules structure is obtained by a ``twisted permutation'' of summands
$$
(sg)(g_i\otimes m)=gg_i\otimes{}^{(gg_i)^{-1}}sm=g_j\otimes({}^{h(gg_i)^{-1}}sh)m=g_j\otimes({}^{g_j^{-1}}sh)m,
$$
if $gg_i=g_jh$ for some $h\in H$.

Parallel to this, if $M'$ is a right $\SH$-module, then the induced right $\SG$-module may be written as
$$
M'\uparrow^{\SG}_{\SH}=\bigoplus_{g'_i\in[H\backslash G]}M\otimes g'_i.
$$
It is a bit surprising, but the reasonable right $\SG$-action is $$
(m'\otimes g'_i)(sg)=m{}^{g'_i}s\otimes g'_ig,
$$
for all $m'\in M', s\in S$ and $g\in G$. This difference attributes to the fact that usually $sg\ne gs$.

Given $g\in[G/H]$ and a $\SH$-module $N$, we put 
$$
{}^g(\SH):=g(\SH)g^{-1}.
$$
The right hand side makes sense because we regard $g$ as an element of $\SG$ and meanwhile $\SH\subset\SG$. It is also a skew group ring, identified with ${}^gS\rtimes{}^gH=S\rtimes{}^gH$ via the following equation 
$$
{}^g(sh)= g(sh)g^{-1}={}^gs{}^gh.
$$
It follows that $g\otimes N$ becomes a ${}^g(\SH)$-module, with
$$
g(sh)g^{-1}(g\otimes n)=g\otimes(sh)n.
$$
Analogous to the situation of group representations, the underlying space of $N$ admits a ${}^g(\SH)$-module structure via the linear isomorphism $N\cong g\otimes N$. We shall denote it by ${}^gN$.

\begin{theorem}[Dade]
Suppose $H, K$ are two subgroups of $G$ and $M\in\SG$-mod. There is a Mackey formula 
$$
M\uparrow_{\SH}^{\SG}\downarrow_{\SK}^{\SG} = \bigoplus_{g\in[K\backslash G/H]}[{}^g(M\downarrow^{\SH}_{S\rtimes(K^g\cap H)})]\uparrow^{\SK}_{S\rtimes(K\cap{}^gH)}.
$$
\end{theorem}

\begin{proof}
This comes from a decomposition of ${}_\SK\SG_\SH$. As a right $\SH$-module, $\SG=\bigoplus_{g_i\in [G/H]} g_i\SH$. Moreover $g_i\SH$, for $g_i\in KgH$, is invariant under the action of $\SK$. The group $K$ acts on $g\SH$ and its stabilizer is
$$
\begin{array}{ll}
\Stab_K(g(\SH))&=\{k\in K \bigm{|} kg(\SH)=g(\SH)\}\\
&=\{k\in K \bigm{|} g^{-1}kg(\SH)=(\SH)\}\\
&=\{k\in K \bigm{|} g^{-1}kg\in H\}\\
&=K\cap{}^gH.
\end{array}
$$
We readily verify that 
$$
\bigoplus_{\substack{g_i\in[G:H]\\g_i\in[K\backslash g/H]}}g_i(\SH)=[g(\SH)]\uparrow^{\SK}_{S\rtimes(K\cap{}^gH)}.
$$ 
However since 
$$
g(\SH)={}^g(\SH\downarrow^{\SH}_{S\rtimes(K^g\cap H)})
$$ 
as a $S\rtimes(K\cap{}^gH)$-module, our decomposition formula follows from it.
\end{proof}

In the proof, we actually showed that $\SG$ is a free right $\SH$-module. It means that $\uparrow^{\SG}_{\SH}$ is exact.

Following Green's approach to group representation theory, Dade introduced a concept of relative projectivity:\\

\noindent Let $M$ be a $\SG$-module and $H$ be a subgroup of $G$. If the $\SG$-module epimorphism
$$
M\downarrow^{\SG}_{\SH}\uparrow^{\SG}_{\SH} \to M
$$
is split, then $M$ is said to be projective relative to $\SH$, or relatively $\SH$-projective.\\

In \cite{Da1, Bo1}, $M$ was called relatively $H$-projective, for the obvious reasons. We use the more cumbersome terminology because, again, we want to be consistent with the rest of the paper. Dade proved several equivalent conditions for $M$ to be relatively $\SH$-projective, which no doubt are parallel to the case of group representations. (Using a relative trace map of Miyashita, defined entirely analogous to the group case, he also obtained a Higman's criterion.)

\begin{theorem}[Dade] Let $R$ be a field of characteristic $p>0$. If $M$ is an indecomposable $\SG$-module, then there is a minimal $p$-subgroup $H$, unique up to conjugacy, such that $M$ is relatively $\SH$-projective. If $K$ is a subgroup such that $M$ is relatively $\SK$-projective, then $\SK$ contains a conjugate of $\SH$.

If $H$ is a Sylow $p$-subgroup of $G$, then every $\SG$-module is relatively $\SH$-projective.
\end{theorem}

Based on the previous theorem, in \cite{Da1, Bo1}, $H$ was called a ``vertex'' of the indecomposable $\SG$-module $M$. However, we only go down to the subalgebra $\SH$, not $RH$. We shall see later on that this is a reason why Dade's construction may be improved for $S=R\calP$.

To study the block theory of a fully group-graded algebra, Boisen introduced the concept of a diagonal subalgebra. We also recall it for skew group algebras. Given $A=\SG$, we set $S^e=S\otimes_RS^{op}$, $G^e=G\times G^{op}$ and $\delta(G)=\{(g,(g^{-1})^{op})\bigm{|} g\in G\}$. Note that $G^e$ may be identified with the product group $G\times G$, and there is a group isomorphism $\delta(G)\cong G$. The group $G^e$ acts on $S^e$ via ${}^{(g,h^{op})}(s,t^{op})=({}^gs,({}^{h^{-1}}t)^{op})$, and the ``diagonal subalgebra'' of $A^e=A\otimes_RA^{op}\cong S^e\rtimes G^e$ is defined to be
$$
\Delta(A)=\bigoplus_{g\in G}A_g\otimes_R(A_{g^{-1}})^{op}=\bigoplus_{g\in G}Sg\otimes_R(g^{-1})^{op}S^{op}\cong S^e\rtimes\delta(G).
$$
As an example, regarding $kG$ as a fully $G$-graded algebra, we have $\Delta(kG)\cong k\delta(G)\cong kG$. Assume $R$ is a field of charcteristic $p>0$. Boisen proved that an indecomposable summand $B$ (a block) of the $A^e$-module $A$ is relatively $S^e\rtimes\delta(H)$-projective, for a minimal $p$-subgroup $H\subset G$. In light of this, he called $H$ a ``defect group'' of $B$. Analogous to the group case, he subsequently defined a generalized Brauer correspondence and established a generalized Brauer's First Main Theorem. As in Dade's treatment, his ``defect'' is also too big when $S=R\calP$.

\section{Transporter categories and their algebras}

We shall study a special class of skew group algebras, namely the transporter category algebras $k\PG$, because $\calP$ has ``local structure''. Like a group algebra or an incidence algebra, $k\PG$ possesses a canonical base, which is the morphism set of the transporter category $\PG$. This basis has an intrinsic structure and is crucial to our theory. Particularly it allows us to introduce transporter subcategories $\QH$ of $\PG$, based on which we will be able to discuss the interactions between representations of $\PG$ and those of $\QH$.

\subsection{Transporter categories and their subcategories} We shall develop the structure theory of transporter categories, before going into their representation theory. We follow standard terminologies in category theory. For a category $\calC$, we show denote by $\Ob\calC$ and $\Mor\calC$ its classes of objects and morphisms. If $\alpha$ and $\beta$ are composable, then we write $\beta\alpha$ for the composite ${\buildrel{\alpha}\over{\to}}{\buildrel{\beta}\over{\to}}$. If $\alpha$ is a morphism, we denote by $s(\alpha)$ and $t(\alpha)$ the \textit{start} (or domain) and the \textit{terminal} (or codomain) of $\alpha$, respectively.

Let $G$ be a group and $\calP$ be a poset. We say $\calP$ admits a $G$-action, or is a $G$-poset, if there exists a group homomorphism $\phi : G\to\Aut(\calP)$. We usually denote by ${}^gx=\phi(g)(x)$ and ${}^g\alpha=\phi(g)(\alpha)$, for all $g\in G, x\in\Ob\calP, \alpha\in\Mor\calP$. The action is trivial if the image $\Im\phi=\Id_{\calP}$.

A group is considered as a category with one object, each group element giving an (auto)morphism, while a poset $\calP$ is a category if $\Ob\calP$ is the underlying set and we regard each $x \le y$ as a morphism $x \to y$. The transporter category on $\calP$ is by definition a Grothendieck construction, some sort of ``semi-direct product'' between two small categories. 

\begin{definition} Let $G$ be a group and $\calP$ be a $G$-poset. Then the \textit{transporter category} $\PG$, of $G$ on $\calP$, is a category, whose objects are the same as those of $\calP$, and whose morphisms are given by $\Hom_{\PG}(x,y)=\{\alpha g\bigm{|} \alpha\in\Hom_{\calP}({}^gx,y)\}$, for any $x, y\in\Ob(\PG)=\Ob\calP$. It is required that $\alpha g=\alpha' g'$ if and only if $\alpha=\alpha'$ and $g=g'$.

If two morphisms $\alpha g$ and $\beta h$ are composable, in the sense that $(\alpha g)(\beta h)\in\Mor(\PG)$, then $(\alpha g)(\beta h)=(\alpha{}^g\beta)(gh)$.  

If $H\subset G$ and $\calQ\subset\calP$ is a $H$-subposet, then we call $\QH$ a \textit{transporter subcategory} of $\PG$.
\end{definition}

Both groups and posets are examples of transporter categories. But we are more interested in the others. Note that when the action of $G$ on $\calP$ is trivial, we simply have $\PG=\calP\times G$.

If $x$ is an object of $\PG$, then we shall use $\langle x\rangle$ to denote the set of objects that are isomorphic to $x$. It is easy to see that $\langle x\rangle=G.x$ is exactly the $G$-orbit containing $x$. Subsequently, $\langle x\rangle\rtimes G$ is a transporter subcategory of $\PG$, and furthermore is a groupoid. The automorphism group $\Aut_{\PG}(x)$ is identified with the stabilizer $G_x=\Stab_G(x)$. It follows that $\{x\}\times G_x$ is a skeleton of $\langle x\rangle\rtimes G$ (hence $\langle x\rangle\rtimes G\simeq \{x\}\times G_x$ as categories).

\begin{example} Let $G=\langle g\bigm{|} g^2=1 \rangle$ and $H$=1. Let $\calP$ be the following $G$-poset
$$
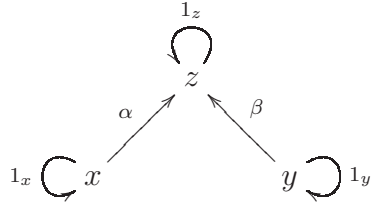
\xymatrix{&z\ar@(ur,ul)_{1_z}&\\
x\ar[ur]^{\alpha}\ar@(ul,dl)_{1_x} & & y\ar[ul]_{\beta}\ar@(ur,dr)^{1_y}}
$$
such that the group generator $g\in G$ fixes $z$ and exchanges $x,y$. On morphisms $g$ acts transitively on the two sets $\{\alpha, \beta\}$, $\{1_x,1_y\}$, and fixes $1_z$. The transporter category $\PG$ is
$$
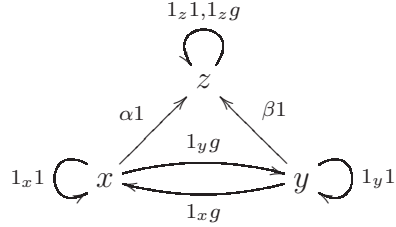
\xymatrix{&z \ar@(ur,ul)_{1_z1,1_zg} &\\
x\ar[ur]^{\alpha 1}\ar@/^/[rr]^{1_yg} \ar@(ul,dl)_{1_x1} & & y\ar[ul]_{\beta 1}\ar@/^/[ll]^{1_xg}\ar@(ur,dr)^{1_y1}}
$$
It is helpful to point out the existence of the following morphisms: $\alpha g=(\alpha 1)(1_x g):y\to z$ and $\beta g=(\beta 1)(1_y g): x\to z$. Choose $\calQ$ to be the subposet consisting of $z$. Then $\QH=\calQ\times H$ is a transporter subcategory consisting of exactly one object $z$ and one morphism $1_z1$.

In $\PG$, the objects $x$ and $y$ are isomorphic. Thus a skeleton of $\PG$ is
$$
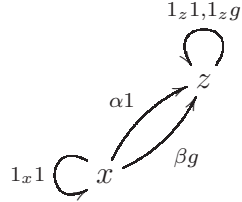
\xymatrix{&z \ar@(ur,ul)_{1_z1,1_zg} &\\
x \ar@/_/[ur]_{\beta g} \ar@/^/[ur]^{\alpha 1} \ar@(ul,dl)_{1_x1} & & }
$$
\end{example}

All transporter categories are \textit{EI categories}, in the sense that every endomorphism is an isomorphism \cite{Xu1}. We shall rely on the EI condition to introduce some crucial constructions. For instance, the EI condition gurantees a partial order on the set of isomorphism classes of objects.

\begin{definition} Let $\calC$ be an EI category and $\calD$ be a full subcategory. Given an object $x\in\Ob\calC$, we define $\calD_{\le x}$ to be the full subcategory of $\calD$ consisting of objects $\{y\in\Ob\calD\bigm{|}\Hom_{\calC}(y,x)\ne\emptyset\}$. Similarly we can define $\calD_{\ge x}$. The subcategory $\calD$ is said to be an \textit{ideal} in $\calC$ if, for every $x\in\Ob\calD$, $\calC_{\le x}\subset\calD$. The subcategory $\calD$ is said to be a \textit{coideal} in $\calC$ if, for every $x\in\Ob\calD$, $\calC_{\ge x}\subset\calD$. The subcategory $\calD$ is said to be \textit{convex} if $\beta\alpha\in\Mor\calD$ implies $t(\alpha)=s(\beta)\in\Ob\calD$. 

We define $\calC_x$ to be the convex subcategory consisting of all objects isomorphic to $x$. Particularly, if $\calC=\PG$ is a transporter category, then $(\PG)_x=\langle x\rangle\rtimes G$.
\end{definition}

Note that ideals and coideals in $\calC$ are always convex. The intersection of two convex (resp. ideal, coideal) subcategories is still convex (resp. ideal, coideal) in $\calC$. These constructions were used to study general EI categories. If $\calC$ happens to be a poset, then every subposet $\calD$ is full as a subcategory. When we deal with transporter categories, we need the following subcategories. By an ideal (or a coideal) $H$-subposet, we mean an ideal (or a coideal) of $\calP$ which is an $H$-subposet of $\calP$ at the same time. For brevity, we shall call an ideal (or a coideal) $H$-subposet an $H$-ideal (or an $H$-coideal).

\begin{definition} Let $\PG$ be a transporter category. If $H\subset G$ and $\calQ\subset\calP$ is an ideal (resp. a coideal) $H$-subposet, then we call $\QH$ a \textit{weak ideal} (resp. a \textit{weak coideal}) of $\PG$. 

If $H\subset G$ and $\calQ\subset\calP$ is a convex $H$-subposet, then we call $\QH$ a \textit{weakly convex transporter subcategory} of $\PG$.
\end{definition}

Weak ideals and coideals are weakly convex. Here $\QH$ is called weakly convex because it is unnecessarily full in $\PG$. We shall demonstrate that weak ideals and coideals, and weakly convex transporter subcategories, which reflect certain local structures of $\PG$, are interesting subjects for investigation.

\begin{lemma} Suppose $\calD$ is a convex subcategory of $\PG$. Then there is a unique convex $G$-subposet $\calQ\subset\calP$ such that $\calD=\calQ\rtimes G$.

In fact, $\calQ\rtimes G$ is convex (or an ideal, or a coideal) if and only if $\calQ$ is convex (or an ideal, or a coideal).
\end{lemma}

\begin{proof} If $x\in\Ob\calD$, then the isomorphism class of $x$ is exactly $G\{x\}\subset\Ob\calD$. Thus the subposet $\calQ=\calP\cap\calD$ is a $G$-subposet of $\calP$. It means that $\calD=\calQ\rtimes G$. Since $\calD$ is convex, $\calQ$ has to be convex in $\calP$.
\end{proof}

The preceding lemma implies, that $\QH$ is weakly convex (resp. weak ideal/coideal) is equivalent to that $\QH$ is a (full) convex (resp. ideal/coideal) subcategory of $\calP\rtimes H$, or that $\calQ\subset\calP$ is convex (resp. ideal/coideal). 

\begin{lemma} Let $\QH$ and $\RK$ be two transporter subcategories of $\PG$. Then $(\QH)\cap(\RK)=(\calQ\cap\calR)\rtimes(H\cap K)$ is a transporter subcategory. 

If both $\QH$ and $\RK$ are weakly convex (resp. weak ideal/coideal), so is $(\calQ\cap\calR)\rtimes(H\cap K)$.
\end{lemma}

\begin{proof} Firstly, the objects of the intersection subcategory form the set $\Ob(\calQ\cap\calR)$. Secondly, any morphism of $\PG$ has a unique way to be written as $\alpha g$, for some $\alpha\in\Mor\calP$ and $g\in G$. Hence if $\alpha g$ belongs to the intersection, we must have $\alpha\in\Mor(\calQ\cap\calR)$ and $g\in H\cap K$. Our first claim follows.

As to the second claim, we see $\calQ\cap\calR$ is convex (resp. an ideal/a coideal) if both $\calQ$ and $\calR$ are.
\end{proof}

We provide methods for constructing weak ideals and co-ideals.

\begin{definition} Let $\calP$ be a $G$-poset. For given subgroups $K\subset H$ and a $K$-subposet $\calQ$, there is a smallest $H$-ideal $\overbrace{{}_H\calQ}$ that contains $\calQ$. We call it the $H$-ideal generated by $\calQ$. Similarly we also have $\underbrace{{}_H\calQ}$ as the $H$-coideal generated by $\calQ$.
\end{definition}

Note that $\underbrace{{}_1\calQ}$ and $\overbrace{{}_1\calQ}$ are simply the smallest coideal and ideal, respectively, that contain $\calQ$. They are often simplified to $\underbrace{\calQ}$ and $\overbrace{\calQ}$. Moreover, if $\calQ$ is a $K$-subposet of $\calP$, then so are $\overbrace{\calQ}$ and $\underbrace{\calQ}$.

From the proposed constructions, we see $\overbrace{{}_H\calQ}\rtimes H$ (resp. $\underbrace{{}_H\calQ}\rtimes H$) is a weak ideal (resp. weak coideal) of $\PG$. We may characterize $\overbrace{{}_H\calQ}\subset\calP$ as follows. Its objects form the set
$$
\{y\in\Ob\calP\bigm{|} \exists\ y\to{}^hx,\ \mbox{for some}\ h\in H, x\in\Ob\calQ\}.
$$
Similarly the objects of $\underbrace{{}_H\calQ}\subset\calP$ form the set
$$
\{y\in\Ob\calP\bigm{|} \exists\  {}^hx\to y,\ \mbox{for some}\ h\in H, x\in\Ob\calQ\}.
$$

As a generalized group, we may define the conjugates of a transporter subcategory of $\PG$.

\begin{definition} Suppose $\QH\subset\PG$ is a transporter subcategory. Given $g\in G$, from $\QH$ we may define another transporter subcategory ${}^g(\QH)$ as follows. Its objects are $\{{}^gx\bigm{|} x\in\Ob(\QH)\}$, and $\Hom_{{}^g(\QH)}({}^gx,{}^gy)=\{{}^g\alpha{}^gh\bigm{|} \alpha h\in\Hom_{\QH}(x,y)\}$. We call ${}^g(\QH)$ a \textit{conjugate} of $\QH$ in $\PG$.
\end{definition}

The conjugate of a transporter subcategory is still a transporter subcategory.

\begin{lemma} Suppose $\QH\subset\PG$ is a transporter subcategory. For each $g\in G$, ${}^g\calQ$ is a ${}^gH$-poset. There is an equality ${}^g(\QH)={}^g\calQ\rtimes{}^gH$. 
\end{lemma}

\begin{proof} The ${}^gH$-action on ${}^g\calQ$ is given by ${}^gx\mapsto {}^{gh}x$ on objects, and ${}^g\alpha\mapsto {}^{gh}\alpha$ on morphisms. 

Since the two categories ${}^g(\QH), {}^g\calQ\rtimes{}^gH$ share the same objects and morphisms, we can identify them.
\end{proof}

The conjugate of a weakly convex transporter subcategory (resp. ideal/coideal) stays weakly convex (resp. ideal/coideal). For brevity, we shall write ${}^g\calQ$ for ${}^g(\calQ\rtimes 1)$. It can be identified with a subposet of $\calP$.

To study representations of transporter categories, it is necessary to generalize some other constructions in group theory.

\begin{definition} Let $\QH$ be a transporter subcategory of $\PG$. We define $N_G(\QH)=\{g\in G \bigm{|} {}^g(\QH)=\QH\}$, called the \textit{normalizer} of $\QH$ in $\PG$. It follows that $\calQ$ is a $N_G(\QH)$-poset.

We also define $C_G(\QH)=\{g\in G \bigm{|} {}^g\alpha{}^gh=\alpha h\ \mbox{for all}\ \alpha h\in\QH\}$, called the \textit{centralizer} of $\QH$ in $\PG$, being a subgroup of $N_G(\QH)$. By definition, $\calQ$ is a $C_G(\QH)$-poset with trivial action, which implies $\calQ\rtimes C_G(\QH)=\calQ\times C_G(\QH)$.
\end{definition}

For brevity, we write $N_G(\calP)$ for $N_G(\calP\rtimes 1)$ and $C_G(\calP)$ for $C_G(\calP\rtimes 1)$. Then we find that $H\subset N_G(\QH)= N_G(\calQ)\cap N_G(H)$ and that $C_G(\QH)= C_G(\calQ)\cap C_G(H)$.

\begin{definition} Let $G$ be a finite group and $\calP$ be a $G$-poset. If $H$ is a $p$-subgroup, for a prime $p$ that divides the order of $G$, then we call $\QH$ a \textit{$p$-transporter subcategory}. Particularly when $S$ is a Sylow $p$-subgroup of $G$, we call $\calP\rtimes S$ a \textit{Sylow $p$-transporter subcategory} of $\PG$.
\end{definition}

The following result justifies our terminology.

\begin{proposition} Let $G$ be a finite group and $\calP\rtimes S$ be a Sylow $p$-transporter subcategory of $\PG$. Then, for each $x\in\Ob(\PG)$, $\Aut_{\calP\rtimes S}(x)$ is a Sylow $p$-subgroup of $\Aut_{\PG}(x)$.
\end{proposition}

\begin{proof} Suppose $S'$ is a Sylow $p$-subgroup of $\Aut_{\PG}(x)\cong\Stab_G(x)$. Then ${}^gS'\subset S$ for some $g\in G$. Since $x\cong{}^gx$ and hence $\Aut_{\PG}({}^gx)\cong\Aut_{\PG}(x)$, ${}^gS'$ is a Sylow $p$-subgroup of $\Aut_{\PG}({}^gx)$. Our first claim follows from the fact that ${}^gS'\subset\Aut_{\calP\rtimes S}({}^gx)\cong\Stab_S(x)$.
\end{proof}

It is easy to see that any two Sylow $p$-transporter subcategories of $\PG$ are conjugate.

\subsection{Enveloping categories and diagonal categories}

These constructions were used in \cite{Xu2}. Let $\calC^e=\calC\times\calC^{op}$ be the product category, between a small category $\calC$ and its opposite, called the \textit{enveloping category}. Given a small category $\calC$, its \textit{category of factorizations} $F(\calC)$ is a small category, whose objects are the morphisms of $\calC$, that is $\Ob F(\calC)=\Mor\calC$. To distinguish, when a morphism $\alpha$ is regarded as an object of $F(\calC)$, we shall use the symbol $[\alpha]$. Let $[\alpha],[\beta]\in \Ob F(\calC)$. There is a morphism from $[\alpha]$ to $[\beta]$ if and only if $\alpha$ is a factor of $\beta$ (as morphisms in $\calC$). More precisely, suppose $\beta=\mu\alpha\gamma$ for $\mu,\gamma\in \Mor\calC$. Then we obtain a morphism $(\mu,\gamma^{op}) : [\alpha] \to [\beta]$ in $F(\calC)$. 

The category of factorization comes with functors $t:F(\calC)\to\calC$, $s:F(\calC)\to\calC^{op}$, and
$$
\delta_{\calC} : F(\calC)\to\calC^e=\calC\times\calC^{op}
$$ 
such that $t([\alpha])=t(\alpha)$, the target of $\alpha$; $s([\alpha])=s(\alpha)$, the source of $\alpha$; $\delta_{\calC}(\alpha)=(t(\alpha),s(\alpha))$ and $\delta_{\calC}(\mu,\gamma^{op})=(\mu,\gamma^{op})$. 

The functor $t: F(\calC)\to\calC$ induces a  homotopy equivalent between classifying spaces $BF(\calC)\simeq B\calC$. See Remark 3.17 for further implications.

If $\calC=G$ is a group, then $F(G)$ is a groupoid, and has its skeleton isomorphic to $G$.

\begin{example} Let $\calP=x{\buildrel{\alpha}\over{\to}} y{\buildrel{\beta}\over{\to}} z$. Then $F(\calP)$ is the poset
$$
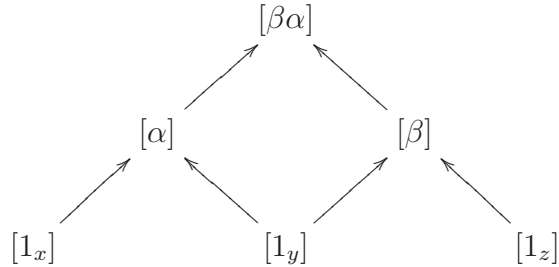
\xymatrix{&&[\beta\alpha]&&\\
&[\alpha]\ar[ur]&&[\beta]\ar[ul]&\\
[1_x] \ar[ur]&&[1_y]\ar[ur] \ar[ul]&&[1_z]\ar[ul]}
$$
It is worth of mentioning that the subposet of $F(\calP)$, with the object $[\beta\alpha]$ removed, is not the factorization category of any subposet of $\calP$. 
\end{example}

\begin{lemma} Let $\PG$ be a transporter category. If $\calQ\subset\calP$ is a subposet, then $F(\calQ)\subset F(\calP)$ is a subposet. Indeed $\calQ$ is convex if and only if $F(\calQ)$ is convex. Thus if $\QH$ is a weakly convex transporter subcategory of $\PG$, then $F(\calQ)\rtimes H\subset F(\calP)\rtimes G$ is a weakly convex as well.
\end{lemma}

\begin{proof} If $\calQ$ is a subposet, then by definition $\Ob F(\calQ)=\Mor\calQ$ is a subset of $\Ob F(\calP)=\Mor\calP$. If $[\alpha],[\beta]\in\Ob F(\calQ)$ and $(u,v) : [\alpha]\to[\beta]$ is a morphism in $F(\calP)$, then $(u,v)\in\Mor F(\calQ)$ because the sources and targets of both $\alpha$ and $\beta$ are in $\calQ$. 

If furthermore $\calQ$ is convex, we easily verify that $F(\calQ)$ is convex. Conversely assume $F(\calQ)$ to be convex. Let $x{\buildrel{\alpha}\over{\to}} y{\buildrel{\beta}\over{\to}} z$ be two morphisms in $\calP$ with $x, z\in\Ob\calQ$. We want to show $y$ belongs to $\calQ$ too. But $1_x$ is a factor of $\alpha$ and $\alpha$ is a factor of $\beta\alpha$. We obtain two morphisms $[1_x]\to[\alpha]\to[\beta\alpha]$ in $\Mor F(\calP)$. Since both $[1_x]$ and $[\beta\alpha]$ belong to $F(\calQ)$, so does $[\alpha]$ which implies that its target belong to $\Ob\calQ$.

Suppose $\QH$ is a weakly convex transporter subcategory. Then $\calQ$ is convex and thus $F(\calQ)$ is convex. Our last claim follows.
\end{proof} 

Suppose $\PG$ is transporter category. Then $G^e$ is a group (which is isomorphic to $G\times G$) and $\calP^e$ admits a $G^e$-action, given by ${}^{(g,h^{op})}(x,y)=({}^gx,{}^{h^{-1}}y)$ on objects, and ${}^{(g,h^{op})}(\alpha,\beta^{op})=({}^g\alpha,({}^{h^{-1}}\beta)^{op})$.

The enveloping category is also a transporter category.

\begin{lemma} There is an isomorphism of categories $\calP^e\rtimes G^e\cong (\PG)^e$.
\end{lemma}

\begin{proof} Since both categories have the same objects as $\calP^e$, we define a functor $\calP^e\rtimes G^e \to (\PG)^e$ to be identity on objects. On morphisms, it is defined by the assignment
$$
(\alpha,\beta^{op})(g,h^{op})\mapsto (\alpha g,({}^h\beta h)^{op}). 
$$
One can readily verify that this functor is an isomorphism between categories.
\end{proof}

The category equivalence $G\to F(G)$ composes with $F(G)\to G^e$ gives the well-known functor $\delta_G: G\to G^e$ (abusing notations). Thus $G$ acts on $\calP^e$ via $\delta_G(G)=\{(g,(g^{-1})^{op})\bigm{|} g\in G\}\subset G^e$. 

\begin{lemma} Let $\calP$ be a $G$-poset. The functor $\delta_{\calP} :F(\calP)\to\calP^e$ is an embedding of posets. Furthermore $F(\calP)$ is a $\delta(G)$-poset. Thus we may identify $F(\calP)$ with a coideal $G$-subposet of $\calP^e$. Consequently, $F(\calP)\rtimes\delta(G)$ is identified with a coideal transporter subcategory of $\calP^e\rtimes 
\delta(G)$.
\end{lemma}

\begin{proof} That $\delta_{\calP} : F(\calP)\to\calP^e$ is an embedding of posets is easy to verify. Indeed $F(\calP)$ is identified with the subposet $\delta_{\calP}(F(\calP))$ of $\calP^e$ consisting of objects $(y,x)$ such that $\Hom_{\calP}(x,y)\ne\emptyset$. The subposet is furthermore a coideal in $\calP^e$. We readily check that $F(\calP)$ can even be regarded as a $\delta(G)$-subposet of $\calP^e$. Then $F(\calP)\rtimes\delta(G)$ is a full subcategory of $\calP^e\rtimes\delta(G)$.
\end{proof}

We emphasize that $F(\calP)$ is not necessarily a $G^e$-poset. Nonetheless, we obtain the following embeddings of transporter categories
$$
F(\calP)\rtimes\delta(G) {\buildrel{\delta_{\calP}}\over{\longrightarrow}} \calP^e\rtimes\delta(G) {\buildrel{\delta_G}\over{\longrightarrow}} \calP^e\rtimes G^e\cong (\PG)^e.
$$
The two categories on the left are weakly convex inside $(\PG)^e$. We shall call $F(\calP)\rtimes\delta(G)$ the \textit{diagonal transporter subcategory} of $(\PG)^e$. It plays a key role in our block theory.

\begin{remark} Let $\PG$ be a transporter category. The functor $t : F(\calP)\to\calP$ induces a homotopy equivalence between the classifying spaces $BF(\calP)\simeq B\calP$, which further gives rise to a homotopy equivalence $B(F(\calP)\rtimes\delta(G))\simeq B(\PG)$, because they are homotopy equivalent to the Borel constructions $EG\times_G BF(\calP)$ and $EG\times_G B\calP$, respectively. It has the consequence that $F(\calP)\rtimes\delta(G)$ and $\PG$ have the same number of connected components.

This fact will be useful in our investigation of block theory.
\end{remark}

\subsection{Category algebras} We want to study the representations of a transporter category. The concept of a category algebra is key to us, see \cite{Xu1} for an account. To this end, we shall recall some basics about category algebras. Let $\calC$ be a small category and $R$ be a commutative ring with identity. Then the category algebra $R\calC$ is defined to be the free $R$-module $R\Mor\calC$, with multiplication determined by composites of morphisms of $\calC$.

\begin{theorem}[Mitchell] Let $\calC$ be a small category such that $\Ob\calC$ is finite. If $R$ is a commutative ring with identity, then $(R\mbox{-mod})^{\calC}\simeq R\calC$-mod.
\end{theorem}

The constant functor $\R : \calC\to R$-mod corresponds to the $R\calC$-module afforded by the free $R$-module $R\Ob\calC$. We shall call $\R$ the \textit{trivial} $k\calC$-module. It plays the role of $R$ in group representations. The trivial module is indecomposable if and only if $\calC$ is connected.

If two small categories are equivalent, then their category algebras are Morita equivalent. Note that, although $BF(\calC)\simeq B\calC$, $RF(\calC)$ is not Morita equivalent to $R\calC$, as $RF(\calC)$ almost always has more simple modules (up to isomorphism). 

It is straightforward to verify that $k\calC^e\cong(k\calC)^e=(k\calC)\otimes_k(k\calC)^{op}$. The ``diagonal subalgebra'' of $(k\PG)^e\cong k\calP^e\rtimes G^e$, in the sense of Boisen, is a transporter category algebra, by the following result.

\begin{lemma} There are isomorphisms of algebras $k\calC^e\cong(k\calC)^e=k\calC\otimes_k(k\calC)^{op}$. In the case of $\calC=\PG$, we have
$$
k\calP^e\rtimes\delta(G) \cong \Delta(k\PG).
$$
\end{lemma}

\begin{proof} The first isomorphism is known and is easy to produce. For the second, we define a map on the base elements
$$
(\alpha\otimes\beta^{op})(g,g^{-1}) \mapsto (\alpha g)\otimes({}^{g^{-1}}\beta g^{-1}),
$$
and then extend it linearly.
\end{proof}

Suppose $\calC$ is an EI category. When $\calD\subset\calC$ is convex, we can regard every $k\calD$-module as a $k\calC$-module. It partially explains why convex subcategories (weakly convex subcategories for a transporter subcategory) play an important role in our theory.

\begin{definition} Let $\calC$ be an EI category. Given a $k\calC$-module $M$, an object $x\in\calC$ is said to be $M$-\textit{minimal} if for any $y\in\Ob\calC$ admitting a non-isomorphism $y\to x$, we must have $M(y)=0$. Similarly, an object $z\in\calC$ is said to be $M$-\textit{maximal} if for any $y\in\Ob\calC$ admitting a non-isomorphism $z\to y$, we must have $M(y)=0$. 

We define $\underbrace{\calC_M}$ to be the coideal whose minimal objects are exactly all the $M$-minimal objects. We also define $\overbrace{\calC_M}$ to be the ideal whose maximal objects are exactly all the $M$-maximal objects. We define the \textit{support} of $M$ to be the full subcategory $\supp M$ consisting of objects $\{x\in\Ob\calC\bigm{|}M(x)\ne 0\}$. We define the \textit{convex support} $\supp^c M$  of $M$ to be $\underbrace{\calC_M}\cap\overbrace{\calC_M}$, which is the convex hull of $\supp M$, the smallest convex subcategory of $\calC$ that contains $\supp M$.
\end{definition}

When there is no confusion, sometimes we call the set $\Ob(\supp M)$ (or $\Ob(\supp^c M)$) the support (or convex support) of $M$. 

\begin{remark} We are interested in the representation theory of $\calC=\PG$. If $N\in k\QH$-mod for a transporter subcategory $\QH\subset\PG$, then, by Lemma 3.5, $\supp N=\calQ_N\rtimes H$ for a unique subposet $\calQ_N$, and $\supp^cN=\calQ^c_N\rtimes H$. Here $\calQ_N=\supp(N\downarrow_{\calQ\rtimes 1})$ and $\calQ^c_N=\supp^c(N\downarrow_{\calQ\rtimes 1})$. Since $\supp N$ (resp. $\supp^c_N$) and $\supp(N\downarrow_{\calQ\rtimes 1})$ (resp. $\supp^c(N\downarrow_{\calQ\rtimes 1})$) determine each other and share the same objects, we may also refer to the latter as the support (resp. the convex support) of $N$.

The two $G$-coideals $\underbrace{{}_G(\calQ_N)}$ and $\underbrace{{}_G(\calQ^c_N)}$ of $\calP$ (see Definition 3.7) are identical, because $\calQ_N$ shares the same minimal objects with $\calQ^c_N$.
\end{remark}

We shall prove that for an indecomposable $k\PG$-module $M$, there is ``no hole'' in its convex support, in the sense that $M(x)\ne 0$ for every $x\in\supp^c M$. In other words, for an indecomposable $M$, $\supp M=\supp^c M$. It explains why we introduce the concept of the convex support, and it will be used when we develop a theory of vertices and sources later on.

\begin{lemma} Suppose $\calP$ is a poset and $M$ is an indecomposable $k\calP$-module. Then for any $x\in\Ob\supp^cM$, $M(x)\ne 0$.
\end{lemma}

\begin{proof} If $\calP$ has several components, then $\supp^cM$ must be contained in exactly one of the components. Without loss of generality, we may assume $\calP$ is connected and suppose $\supp^c M=\calP$. Assume there is some $x\in\Ob\calP$ such that $M(x)=0$. By the assumption, $x$ can not be either maximal or minimal. We may consider its projective cover $P_M$. Let $\pi: P_M \to M$ be the epimorphism. Then $(\ker\pi)(x)=P_M(x)$. But $P_M\cong \oplus_iP_{y_i}^{n_i}$ for $M$-minimal objects $y_i$ (all satisfying $y_i\le x$). Since, in a poset, between any two objects there is at most one morphism, it forces $(\ker\pi)(z)=P_M(z)$ for all $z\ge x$. In turn it implies $M(z)=0$ for all $z\ge x$, a contradiction to $\supp^cM=\calP$.
\end{proof}

\begin{proposition} Let $\PG$ be a connected transporter category and $M$ an indecomposable $k\PG$-module. Then for any $x\in\supp^c M$, $M(x)\ne 0$. Thus $\supp M=\supp^cM$.
\end{proposition}

\begin{proof} Without loss of generality, suppose $\supp^c M=\PG$. Then the restriction decomposes as a direct sum of indecomposable $k\calP\rtimes 1$-modules
$$
M\downarrow^{\PG}_{\calP\rtimes 1}=\bigoplus_{i=1}^n M_i.
$$
The group $G$ acts on the set $\{M_1,\cdots,M_n\}$ of indecomposable $k\calP\rtimes 1$-modules. The $k\PG$-module $M$ is indecomposable if and only if there is only one $G$-orbit on the set $\{M_1,\cdots,M_n\}$. These $M_i$'s have conjugate supports.

Now assume $M(x)=0$ ($x$ not maximal or minimal, by assumption). It implies that $M(x')=0$ and hence $M_i(x')=0, \forall i$, for all $x'\cong x$ in $\Ob(\PG)$. Thus if there exists a morphism $x\to z$ in $\Mor(\PG)$, there exists $x'\to z$ in $\Mor\calP$ for some $x'\cong x$. Applying Lemma 3.22, we find that $M_i(z)=0$ for all $i$. It means that $M(z)=0$ for every $z$ with a morphism $x\to z$, a contradiction to our assumption that $\supp^c M=\PG$.
\end{proof}

By the definitions of limits, it is straightforward to prove the following isomorphisms.

\begin{lemma} Suppose $\calC$ is finite EI and $\calD\subset\calC$ is a full subcategory. Let $M\in k\calC$-mod. Then
\begin{enumerate}
\item if $\calD$ is a coideal and contains $\supp M$, $\lim_{\calC}M\cong\lim_{\calD}M\downarrow_{\calD}$; and 

\item if $\calD$ is an ideal and contains $\supp M$, $\colim_{\calC}M\cong\colim_{\calD}M\downarrow_{\calD}$.
\end{enumerate}
\end{lemma}

\section{Vertices and Sources}

In Section 2, we briefly explained, given a skew group algebra $\SG$, how Dade conceived a theory of vertices and sources, through comparing it with various subalgebras $\SH$, for $H\subset G$. When $\calP=G/H$ with left $G$-multiplication, $k\PG$ is Morita equivalent to $kH$. If we take the trivial module $\k\in k\PG$-mod (corresponding to $k\in kH$-mod), a ``vertex'' of $\k$, in the sense of Dade, would be $T$, a Sylow $p$-subgroup of $H$. However, it actually means that $\k$ is projective relative to $k\calP\rtimes T$, which is \textit{not} Morita equivalent to $kT$, and thus does not match the classical theory for group algebras. See Example 4.23 for more details. We shall fix the problem and get a sharpened definition.

\subsection{Inclusions and restrictions} Let $\calP$ be a $G$-poset. Suppose $H$ is a subgroup of $G$. We may regard $\calP$ as an $H$-poset. Assume $\calQ$ is a $H$-subposet of $\calP$. Then we have two faithful functors and their composite, which are inclusions of transporter categories,
$$
\iota^{\PG}_{\QH}=\iota^{\PG}_{\calP\rtimes H}\iota^{\calP\rtimes H}_{\QH} : \calQ\rtimes H {\buildrel{\iota^{\calP\rtimes H}_{\QH}}\over{\longrightarrow}} \calP\rtimes H {\buildrel{\iota^{\PG}_{\calP\rtimes H}}\over{\longrightarrow}} \PG.
$$
We shall study their effects on representations of these categories.

If $\calP'$ happens to be a $G$-subposet of $\calP$, there are two similar faithful functors and their composite as follows
$$
\iota^{\PG}_{\calP'\rtimes G}\iota^{\calP'\rtimes G}_{\calP'\rtimes H} : \calP'\rtimes H {\buildrel{\iota^{\calP'\rtimes G}_{\calP'\rtimes H}}\over{\longrightarrow}} \calP'\rtimes G {\buildrel{\iota^{\PG}_{\calP'\rtimes G}}\over{\longrightarrow}} \PG.
$$
Obviously in this case, $\iota^{\PG}_{\calP'\rtimes G}\iota^{\calP'\rtimes G}_{\calP'\rtimes H}=\iota^{\PG}_{\calP\rtimes H}\iota^{\calP\rtimes H}_{\calP'\rtimes H}$.

In general, let $\calD$ and $\calC$ be two small categories and $\tau : \calD\to\calC$ be a functor. Then $\tau$ induces a restriction along $\tau$, written as $\Res_{\tau} : k\calC$-mod $\to k\calD$-mod. When $\calC=G$ is a group and $\calD=H$ is a subgroup, the restriction along the inclusion is the usual restriction $\downarrow^G_H$. In the present paper, we will chiefly be interested in the situation where $\tau$ is an inclusion. When this is the case, we shall denote $\Res_{\tau}M$ by $M\downarrow^{\calC}_{\calD}$, for all $M\in k\calC$-mod. Let $\PG$ be a transporter category, and $\QH$ be a transporter subcategory. We have 
$$
M\downarrow^{\PG}_{\QH}=M\downarrow^{\PG}_{\calP\rtimes H}\downarrow^{\calP\rtimes H}_{\QH},
$$
for all $M\in k\PG$-mod. We comment that the restriction $\downarrow^{\calP\rtimes H}_{\QH}=1_{k\QH}\cdot -$ is a brutal truncation, not coming from a unital algebra homomorphism. 

\subsection{Kan extensions} In group representations, the restriction has isomorphic left and right adjoints, called the \textit{induction} and the \textit{co-induction}. They are actually special cases of the Kan extensions. Let $\calD$ and $\calC$ be two small categories and $\tau : \calD\to\calC$ be a functor. Then the restriction $\Res_{\tau} : k\calC$-mod $\to k\calD$-mod possesses both left and right adjoints (Kan extensions) $LK_{\tau},RK_{\tau} : k\calD$-mod $\to k\calC$-mod. These are well-known constructions in homological algebra. However they are seldom used in representation theory since they are usually extremely hard to compute. We shall see, in the representation theory of transporter categories, that it is possible to understand the Kan extensions and then to apply these constructions. 

By definition, given $N\in k\calD$-mod, we can construct two $k\calC$-modules, $LK_{\tau}N$ and $RK_{\tau}N$. Suppose $x\in\Ob\calC$. Then
$$
[LK_{\tau}N](x)=\lim_{\tau/x}\tilde{N}\ \ \mbox{and}\ \ [RK_{\tau}N](x)=\colim_{x\backslash\tau}\bar{N}.
$$
Here $\tau/x$ and $x\backslash\tau$ are categories over and under $x$, respectively. We often just refer to them as an \textit{overcategory} or an \textit{undercategory}. The objects of $\tau/x$ are pairs $(y,\alpha)$, where $y\in\Ob\calD$ and $\alpha\in\Hom_{\calC}(\tau(y),x)$; while a morphism $f : (y,\alpha)\to(z,\beta)$ is a morphism $f\in\Hom_{\calD}(y,z)$, such that $\alpha=\beta\tau(f)$. By comparison, the objects of $x\backslash\tau$ are pairs $(\gamma,w)$, with $w\in\Ob\calD$ and $\gamma\in\Hom_{\calC}(x,\tau(w))$. The morphisms are defined accordingly. Meanwhile $\tilde{N}$ is the restriction along the canonical functor (a projection) $\tau/x \to \calD$ and $\bar{N}$ is the restriction along $x\backslash\tau\to\calD$. For convenience, we shall abbreviate the defining formulas of Kan extensions to
$$
[LK_{\tau}N](x)=\lim_{\tau/x}N\ \ \mbox{and}\ \ [RK_{\tau}N](x)=\colim_{x\backslash\tau}N,
$$
in the rest of the present paper, despite the fact that $N$ is not defined on the over- and undercategories.

When $\calC=G$ is a group and $\calD=H$ is a subgroup, the three functors associated to the inclusion are the usual restriction $\Res_{\iota}=\downarrow^G_H$, induction $LK_{\iota}=\uparrow^G_H$ and coinduction $RK_{\iota}=\Uparrow^G_H$ in group representations. We emphasize that in general $LK_{\tau}\not\cong RK_{\tau}$. The computations of these functors usually amounts to analyzing the structures of all relevant over- and undercategories.

To be consistent, we shall write $M\downarrow^{\calC}_{\calD}=\Res_{\tau}M$ for all $M\in k\calC$-mod, and $N\uparrow^{\calC}_{\calD}=LK_{\tau}N$ and $N\Uparrow^{\calC}_{\calD}=RK_{\tau}N$, for all $N\in k\calD$-mod.

Let $\PG$ be a transporter category, and $\QH$ be a transporter subcategory. By earlier discussions, we have 
$$
\downarrow^{\PG}_{\QH}=\downarrow^{\PG}_{\calP\rtimes H}\downarrow^{\calP\rtimes H}_{\QH}.
$$
Consequently the left and right adjoints of $\downarrow^{\PG}_{\QH}$ satisfy 
$$
\uparrow^{\PG}_{\QH}=\uparrow^{\calP\rtimes H}_{\QH}\uparrow^{\PG}_{\calP\rtimes H}
$$ 
and 
$$
\Uparrow^{\PG}_{\QH}=\Uparrow^{\calP\rtimes H}_{\QH}\Uparrow^{\PG}_{\calP\rtimes H}.
$$

At this point, we shall compute several Kan extensions in the context of transporter categories. This will be a cornerstone for our upcoming developments.

\begin{theorem} Let $\PG$ be a transporter category, and $\QH$ be a transporter subcategory. Suppose $\iota=\iota^{\PG}_{\QH}$ is the inclusion functor. Choose a set of left coset representatives $[G/H]=\{g_1,\cdots,g_n\}$.
\begin{enumerate}
\item For each $x\in\Ob(\PG)$, $\iota/x$ (resp. $x\backslash\iota$), if not empty, is the disjoint union of full subcategories $\iota/x=\coprod_{g_i\in[G/H]}(\iota/x)_i$ (resp. $x\backslash\iota=\coprod_{g_i\in[G/H]}(x\backslash\iota)_i$). Moreover each $(\iota/x)_i$ (resp. $(x\backslash\iota)_i$) has its skeleton isomorphic to $({}^{g_i}\calQ)_{\le x}$ (resp. $({}^{g_i}\calQ)_{\ge x}$). Consequently, given a $k\QH$-module $N$,
$$
[N\uparrow^{\PG}_{\QH}](x)=\bigoplus_{g_i\in[G/H]}\lim_{(\iota/x)_i}N\cong\bigoplus_{g_i\in[G/H]}\lim_{({}^{g_i}\calQ)_{\le x}}N.
$$
(resp. 
$$
[N\Uparrow^{\PG}_{\QH}](x)=\bigoplus_{g_i\in[G/H]}\colim_{(x\backslash\iota)_i}N\cong\bigoplus_{g_i\in[G/H]}\colim_{({}^{g_i}\calQ)_{\ge x}}N.)
$$
If $\beta s\in\Hom_{\PG}(x,z)$, then it induces a functor $\iota/x\to\iota/z$ (resp. $z\backslash\iota \to x\backslash\iota$), given by $(y,\alpha g)\mapsto (y,\beta s\alpha g)=(y,\beta{}^s\alpha sg)$ (resp. $(\alpha' g', y')\mapsto (\alpha' g'\beta s,y')=(\alpha'{}^{g'}\beta g's,y')$), which defines a map
$$
[N\uparrow^{\PG}_{\QH}](\beta s) : [N\uparrow^{\PG}_{\QH}](x) \to [N\uparrow^{\PG}_{\QH}](z)
$$
(resp.
$$
[N\Uparrow^{\PG}_{\QH}](\beta s) : [N\Uparrow^{\PG}_{\QH}](x) \to [N\Uparrow^{\PG}_{\QH}](z))
$$
that gives rise to the $\beta s$-action on $N\uparrow^{\PG}_{\QH}$ (resp. $N\Uparrow^{\PG}_{\QH}$).

\item The left and right adjoints of $\downarrow^{\PG}_{\calP\rtimes H}$ are $\uparrow^{\PG}_{\calP\rtimes H}\cong k\PG\otimes_{k\calP\rtimes H}-$ and $\Uparrow^{\PG}_{\calP\rtimes H}\cong \Hom_{k\calP\rtimes H}(k\PG,-)$, respectively.

\item The left and right adjoints of $\downarrow^{\calP\rtimes H}_{\QH}$ are $\uparrow^{\calP\rtimes H}_{\QH}\cong\lim_{\calQ_{\le}}$ and $\Uparrow^{\calP\rtimes H}_{\QH}\cong\colim_{\calQ_{\ge}}$, respectively.
\end{enumerate}
\end{theorem}

\begin{proof} We shall see that (2) and (3) are special cases of (1), although (2) can be established by classical constructions. We will only compute the left Kan extensions, and leave the right Kan extensions to the interested reader to check.

In order to prove (1), we first abbreviate the inclusion functor to $\iota$. Suppose $\iota/x$ is not empty. Let $(y,\alpha s)$ be an object of $\iota/x$. Choose $g_1,\cdots,g_n\in [G/H]$ to be a set of left coset representatives. Then $s=g_ih$ for some $g_i$ and $h\in H$. We see that $({}^hy,\alpha g_i)$ also belongs to $\iota/x$ and it is isomorphic to $(y,\alpha s)$ in $\iota/x$. Up to isomorphism, every object of $\iota/x$ is of the form $(y',\alpha g_i)$ for some $y'\in\Ob(\QH)$, $\alpha\in\Mor\calP$ and $g_i\in[G/H]$. Moreover if $\gamma h:(y_1,\alpha_1 g_i)\to (y_2,\alpha_2 g_j)$ is a morphism, we obtain equalities $\alpha_1=\alpha_2{}^{g_j}\gamma$ and $g_i=g_jh$. The latter implies $g_iH=g_jH$. Particularly, it tells us that two objects $(y, \alpha s)$, $(z,\beta t)$ of $\iota/x$ are connected by a zigzag of morphisms, or lie in the same connected component, if and only if $sH=tH$. Moreover, since both $h$ and $\gamma$ are uniquely determined, it also tells us that, between any two objects of $\iota/x$, there is at most one morphism. Thus the skeleton of $\iota/x$ must be a poset, with up to $|G:H|$ connected components.

Denote by $(\iota/x)_i$ the subcategory of $\iota/x$ consisting of objects of the form $(y,\alpha g_ih)$, where $h\in H$. Our calculation means that $\iota/x$ is the disjoint union of $(\iota/x)_i$, each indexed by a left coset representative $g_i\in[G/H]$.

Now we define a functor, between posets, $(\iota/x)_i \to ({}^{g_i}\calQ)_{\le x}$ by $(y,\alpha g_ih)\mapsto {}^{g_ih}y$ on objects. This functor has a quasi-inverse, given on objects by $z\mapsto ({}^{g_i^{-1}}z,\beta g_i)$, if $z\in {}^{g_i}\calQ$ and $\Hom_{\calP}(z,x)=\{\beta\}$. Hence $(\iota/x)_i\simeq({}^{g_i}\calQ)_{\le x}$.

As to (2), since this is a special case of $\calQ=\calP$ in (1), we find that $\iota^{\PG}_{\calP\rtimes H}\simeq\coprod_{g_i\in[G/H]}({}^{g_i}\calP)_{\le x}$. In this case each $({}^{g_i}\calP)_{\le x}$ has terminal objects $\{({}^{(g_ih)^{-1}}x,1_xg_ih)\bigm{|} h\in H\}$. Given that $g_j\in [G/H]$ is the unique coset representative satisfying $gg_i\in g_jH$, which sends the terminal object $({}^{g_i^{-1}}x,1_xg_i)$ to the terminal object $({}^{(gg_i)^{-1}}({}^gx),1_{{}^gx}gg_i)\cong({}^{g_j^{-1}}({}^gx),1_{{}^gx}g_j)$. Thus
$$
[N\uparrow^{\PG}_{\calP\rtimes H}](x)=\lim_{\iota^{\PG}_{\calP\rtimes H}/x}N\cong\bigoplus_{g_i\in [G/H]}N({}^{g_i^{-1}}x),
$$
and the $G$-action is determined by $N({}^{g_i^{-1}}x)\to N({}^{g_j^{-1}}({}^gx))\cong N({}^{g_j^{-1}}x)$. Meanwhile, the restriction $\downarrow^{\PG}_{\calP\rtimes H}$ is induced by the injective unital algebra homomorphism $k\calP\rtimes H \to k\PG$. So are its adjoints. The functor that we just built is isomorphic to $k\PG\otimes_{k\calP\rtimes H}-$, identified with $\uparrow^G_H$, in Section 2, used by Boisen and Dade.

We turn to (3). By (1), $\iota^{\calP\rtimes H}_{\QH}/x\simeq\calQ_{\le x}$, for each $x\in\Ob(\PG)$. It follows that 
$$
N\uparrow^{\calP\rtimes H}_{\QH}(x)=[LK_{\iota^{\calP\rtimes H}_{\QH}}N](x)=\lim_{\iota^{\calP\rtimes H}_{\QH}/x}N\cong\lim_{\calQ_{\le x}}N.
$$
\end{proof}

For the interested reader, when analysing $x\backslash\iota\ne\emptyset$, we should notice that, for each $(\alpha s,y)\in\Ob(x\backslash\iota)$, there is an isomorphism $(\alpha s,y)\cong ({}^{s^{-1}}\alpha,{}^{s^{-1}}y)$. The left coset representatives provide a set of right representatives $[H\backslash G]=\{g^{-1}_i\}$. Since $s=hg^{-1}_i$ for a unique $i$, we obtain an isomorphism $(\alpha s,y)\cong ({}^{h^{-1}}\alpha g^{-1}_i,{}^{{h^{-1}}}y)$ in $x\backslash\iota$. Hence two objects $(\alpha s, y)$ and $(\beta t, z)$, of $x\backslash\iota$, belong to the same connected component, if and only if $Hs=Ht$. We can also verify that $x\backslash\iota$ has a poset as its skeleton. Denote by $(x\backslash\iota)_i$ the connected component indexed by $g^{-1}_i$ such that $x\backslash\iota=\prod_{g^{-1}_i\in[H\backslash G]}(x\backslash\iota)_i$. It follows that the functor $(x\backslash\iota)_i \to ({}^{g_i}\calQ)_{\ge x}$, defined on objects by $(\alpha hg^{-1}_i,y)\to {}^{g_ih^{-1}}y$ is an equivalence of categories.

Occasionally we will deal with the case where $\calP'$ is a $G$-subposet of $\calP$. Under the circumstance, we will have $\downarrow^{\PG}_{\calP'\rtimes G}\downarrow^{\calP'\rtimes G}_{\calP'\rtimes H}=\downarrow^{\PG}_{\calP\rtimes H}\downarrow^{\calP\rtimes H}_{\calP'\rtimes H}$, $\uparrow^{\calP'\rtimes G}_{\calP'\rtimes H}\uparrow^{\PG}_{\calP'\rtimes G}=\uparrow^{\calP\rtimes H}_{\calP'\rtimes H}\uparrow^{\PG}_{\calP\rtimes H}$, and $\Uparrow^{\calP'\rtimes G}_{\calP'\rtimes H}\Uparrow^{\PG}_{\calP'\rtimes G}=\Uparrow^{\calP\rtimes H}_{\calP'\rtimes H}\Uparrow^{\PG}_{\calP\rtimes H}$.

From our proof of the above theorem, if the (convex) support of an indecomposable $k\calP\rtimes H$-module $L$ is $\QH$, then the support of $L\uparrow^{\PG}_{\calP\rtimes H}$ can be larger, and contains ${}^g(\QH)$, $\forall g\in G$.

\begin{definition} Suppose $\QH$ is a transporter subcategory of $\PG$. Let $N\in k\QH$-mod. Fix an element $g\in G$. We define a $k {}^g(\QH)$-module ${}^gN$ as follows. It equals $N$ as a vector space, with ${}^gN({}^gx)=N(x)$ for each $x\in\Ob(\QH)$. While for $n\in {}^gN({}^gx), {}^gx\in\Ob{}^g(\QH)$, and ${}^g(\alpha h) : {}^gx\to {}^gy$, we define $({}^g(\alpha h))(n)=(\alpha h)(n)$. We shall call ${}^gN$ a \textit{conjugate} of $N$.

If $g\in N_G(\QH)$, then ${}^gN$ is a $k\QH$-module.
\end{definition}

\begin{corollary} Let $\PG$ be a transporter category and $\QH$ be a transporter subcategory.

\begin{enumerate}
\item Let $N\in k\QH$-mod. Then there is a split surjection
$$
N\uparrow^{\PG}_{\QH}\downarrow^{\PG}_{\calP\rtimes H}\cong\bigoplus_{g_i\in[G/H]}{}^{g_i}(N\uparrow^{\calP\rtimes H}_{\QH}).
$$ 

\item Let $N\in k\QH$-mod. Then there is a split surjection
$$
p: N\uparrow^{\PG}_{\QH}\downarrow^{\PG}_{\QH}\to N.
$$ 
In particular if $N'\in k\calP\rtimes H$-mod, as $k\calP\rtimes H$-modules every summand of $N'\uparrow^{\PG}_{\calP\rtimes H}\downarrow^{\PG}_{\calP\rtimes H}$ is isomorphic to $N'$.

\item If $N$ is an indecomposable $k\QH$-module with support $\QH$, then $N\uparrow_{\QH}^{\calP\rtimes H}$ is an indecomposable $k\calP\rtimes H$-module with support $\underbrace{\calQ}\rtimes H$.

\item If $L$ is an indecomposable $k\calP\rtimes H$-module with $\supp L=\calP\rtimes H$, then every indecomposable summand of $L\uparrow^{\PG}_{\calP\rtimes H}$ has $\PG$ as its convex support.
\end{enumerate}
\end{corollary}

\begin{proof} To prove the first statement, we write
$$
N\uparrow^{\PG}_{\QH}=(N\uparrow^{\calP\rtimes H}_{\QH})\uparrow^{\PG}_{\calP\rtimes H},
$$ 
and then use Theorem 4.2 (2).

The second statement follows from the first, on restriction further down to $\QH$. The direct summand corresponding to $g_i=1$ is a copy of $N$. 

For 3), since $\QH$ is a full subcategory of $\calP\rtimes H$, $N\uparrow^{\calP\rtimes H}_{\QH}$ is indecomposable. By $N\uparrow^{\calP\rtimes H}_{\QH}=N\uparrow^{\calP}_{\calQ}$ and $\supp N=\supp^cN=\QH$, we know $N\uparrow^{\calP\rtimes H}_{\QH}$ has $\underbrace{\calQ}\rtimes H$ as its support.

If $M$ is an indecomposable summand of $L\uparrow^{\PG}_{\calP\rtimes H}$, then $M\downarrow_{\calP\rtimes H}$ is a direct summand of $L\uparrow^{\PG}_{\calP\rtimes H}\downarrow^{\PG}_{\calP\rtimes H}\cong L^{|G:H|}$. Thus every summand of $M\downarrow_{\calP\rtimes H}$ is isomorphic to $L$. Then 4) follows from it.
\end{proof}

Note that the last statement may not be true if $\supp L\ne \calP\rtimes H$. 

The first statement can be regarded as a generalization of a standard result in group representations. However, usually the direct summands of $N\uparrow^{\PG}_{\QH}\downarrow^{\PG}_{\QH}$ need not be isomorphic to each other.

\begin{example} Consider the transporter category $\PG$ in Example 3.2
$$
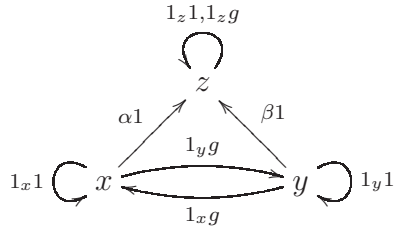
\xymatrix{&z \ar@(ur,ul)_{1_z1,1_zg} &\\
x\ar[ur]^{\alpha 1}\ar@/^/[rr]^{1_yg} \ar@(ul,dl)_{1_x1} & & y\ar[ul]_{\beta 1}\ar@/^/[ll]^{1_xg}\ar@(ur,dr)^{1_y1}}
$$

Suppose $K=1$ is the trivial subgroup of $G$ and $\calR=\{x\}$ is the subposet consisting of a single object $x$. Then $\RK=\{x\}\times 1$. Let $\k_x$ be the trivial $k\{x\}\times 1$-module. It can also regarded as an (atomic) $k\calP\times 1$-module.

\begin{enumerate}
\item The induced module $\k_x\uparrow^{\PG}_{\{x\}\times 1}$ is given by $\k_x\uparrow^{\PG}_{\{x\}\times 1}(x)\cong \k_x\uparrow^{\PG}_{\{x\}\times 1}(y)\cong k$ as vector spaces and $\k_x\uparrow^{\PG}_{\{x\}\times 1}(z)=k^2$ because $\iota/z$ is the disjoint union of two trivial posets $\{(y,\beta 1)\}$ and $\{(y,\alpha g)\}$. One may check that $\k_x\uparrow^{\PG}_{\{x\}\times 1}$ is indecomposable. Moreover $\k_x\uparrow^{\PG}_{\{x\}\times 1}=(\k_x\uparrow^{\calP\times 1}_{\{x\}\times 1})\uparrow^{\PG}_{\calP\times 1}$, with $\k_x\uparrow^{\calP\times 1}_{\{x\}\times 1}(y)=0$ and $\k_x\uparrow^{\calP\times 1}_{\{x\}\times 1}(z)=k$.

\item By comparison, $\k_x\uparrow^{\PG}_{\calP\times 1}$ is given by $\k_x\uparrow^{\PG}_{\calP\times 1}(x)\cong \k_x\uparrow^{\PG}_{\calP\times 1}(y)\cong k$ as vector spaces and $\k_x\uparrow^{\PG}_{\calP\times 1}(z)=0$.

\item On restriction to $\calP\times 1$, $\k_x\uparrow^{\PG}_{\{x\}\times 1}\downarrow^{\PG}_{\calP\times 1}\cong \k_{\calQ_x}\oplus \k_{\calQ_y}$, where $\calQ_{x}=x\to z$ and $\calQ_y=y\to z$, satisfying ${}^g\calQ_x=\calQ_y$. Here $\k_{\calQ_x}$ and $\k_{\calQ_y}$ are the trivial $k\calQ_x\times 1$- and $k\calQ_y\times 1$-modules, considered as indecomposable $k\calP\times 1$-module. Note that ${}^g\k_{\calQ_x}\cong\k_{\calQ_y}$.

\item On restriction to $\{x\}\times 1$, $\k_x\uparrow^{\PG}_{\{x\}\times 1}\downarrow^{\PG}_{\{x\}\times 1}\cong \k_x$.
\end{enumerate}
\end{example}

In the end, we record some technical statements that we need in proving the Mackey formula.

\begin{corollary} Let $\PG$ be a transporter category, and $\calQ$ be a $G$-subposet of $\calP$. Also let $H$ be a subgroup of $G$ and $\RK$ be a transporter subcategory of $\PG$. Then for every $M\in k\calP\rtimes H$-mod and $N\in k\calQ\rtimes H$-mod, there exist isomorphisms
\begin{enumerate}
\item $M\uparrow^{\PG}_{\calP\rtimes H}\downarrow^{\PG}_{\calQ\rtimes G}\cong M\downarrow^{\calP\rtimes H}_{\calQ\rtimes H}\uparrow^{\calQ\rtimes G}_{\QH}$;
\item $N\uparrow^{\calP\rtimes H}_{\QH}\downarrow^{\calP\rtimes H}_{\calP\rtimes(K\cap H)}\cong N\downarrow^{\QH}_{\calQ\rtimes(K\cap H)}\uparrow^{\calP\rtimes(K\cap H)}_{\calQ\rtimes(K\cap H)}$; and
\item $N\uparrow^{\calP\rtimes H}_{\QH}\downarrow^{\calP\rtimes H}_{\calR\rtimes(K\cap H)}\cong N\downarrow^{\QH}_{\calQ\rtimes(K\cap H)}\uparrow^{\calP\rtimes(K\cap H)}_{\calQ\rtimes(K\cap H)}\downarrow^{\calP\rtimes(K\cap H)}_{\calR\rtimes(K\cap H)}$.
\end{enumerate}
\end{corollary}

\begin{proof} For (1), we have 
$$
\begin{array}{ll}
M\uparrow^{\PG}_{\calP\rtimes H}\downarrow^{\PG}_{\calQ\rtimes G}(x)&=\bigoplus_{g_i\in[G/H]}\lim_{\calP_{\le x}}M\\
&\\
&\cong\bigoplus_{g_i\in[G/H]}M(x)\\
&\\
&=\bigoplus_{g_i\in[G/H]}[M\downarrow^{\calP\rtimes H}_{\calQ\rtimes H}(x)]\\
&\\
&\cong\bigoplus_{g_i\in[G/H]}\lim_{\calQ_{\le x}}M\downarrow^{\calP\rtimes H}_{\calQ\rtimes H}\\
&\\
&=[M\downarrow^{\calP\rtimes H}_{\calQ\rtimes H}\uparrow^{\calQ\rtimes G}_{\QH}](x),
\end{array}
$$
for each $x\in\Ob(\calQ\rtimes G)$. Here we used the fact that $x$ is the terminal object of both $\calP_{\le x}$ and $\calQ_{\le x}$. Then one may readily check that these isomorphisms assemble to a module isomorphism.

As for (2), it comes from Theorem 4.1 (2) and the isomorphism
$$
[\lim_{\calQ_{\le}}N]\downarrow^{\calP\rtimes H}_{\calP\rtimes(K\cap H)}\cong\lim_{\calQ_{\le}}[N\downarrow^{\calP\rtimes H}_{\calP\rtimes(K\cap H)}].
$$
Now (3) follows from (2), because
$$
\begin{array}{ll}
N\uparrow^{\calP\rtimes H}_{\QH}\downarrow^{\calP\rtimes H}_{\calR\rtimes(K\cap H)}&=[N\uparrow^{\calP\rtimes H}_{\QH}\downarrow^{\calP\rtimes H}_{\calP\rtimes(K\cap H)}]\downarrow^{\calP\rtimes(K\cap H)}_{\calR\rtimes(K\cap H)}\\
&\\
&\cong [N\downarrow^{\QH}_{\calQ\rtimes(K\cap H)}\uparrow^{\calP\rtimes(K\cap H)}_{\calQ\rtimes(K\cap H)}]\downarrow^{\calP\rtimes(K\cap H)}_{\calR\rtimes(K\cap H)}.
\end{array}
$$
\end{proof}

\subsection{Relative projectivity} We want to develop a theory of vertices and sources for transporter category algebras. It will generalize the original theory for group algebras, when we regard groups as transporter categories. More precisely, let $\PG$ be a transporter category and $M$ be an indecomposable $k\PG$-module. We shall define a vertex $\calV_M$, of $M$, to be a weakly convex transporter subcategory $\QH$, unique up to conjugacy in $\PG$, and its source to be an indecomposable $k\QH$-module, which is also unique up to conjugacy, such that $M\bigm{|} N\uparrow^{\PG}_{\QH}$.

Our generalization is motivated by two existing theories, one for fully group-graded algebras \cite{Da1, Bo1}, and the other for EI category algebras \cite{Xu1}. It relies on the observation that transporter category algebras are both fully group-graded algebras and EI category algebras.

Suppose $\calC$ is a finite category and $\calE$ is a subcategory. Then the co-unit of the adjunction between $\downarrow^{\calC}_{\calE}$ and $\uparrow^{\calC}_{\calE}$ gives rise to a canonical map $\epsilon_M : M\downarrow^{\calC}_{\calE}\uparrow^{\calC}_{\calE} \to M$ for each $M\in k\calC$-mod. In general, $\epsilon_M$ is not surjective, and one may easily construct examples in which $\epsilon_M=0$.

From now on, we shall assume $\calC$ to be finite EI. We want to recall a fraction of the theory of vertices and sources for EI category algebras \cite{Xu1}. In fact, in order to get better results, we must modify and improve it. First of all, we shall see $\epsilon_M$ can be surjective for some convex subcategory $\calE\subset\calC$.

\begin{lemma} The canonical map $M\downarrow^{\calC}_{\calE}\uparrow^{\calC}_{\calE} \to M$ is surjective if $\calE$ contains $\supp^cM$.
\end{lemma}

\begin{proof} We may assume without loss of generality that $\calE=\supp^cM$. Then it follows from a basic property of the Kan extensions that the counit gives an isomorphism $M\downarrow^{\calC}_{\calE}\uparrow^{\calC}_{\calE}(x) \cong M(x)$, on every $x\in\Ob\calE$.
\end{proof}

In light of this lemma, we restrict $M$ to its convex support in order to sharpen the vertices (to non-full subcategories) of $M$ given in \cite{Xu1}. This restriction will not change the nature of our discussion, as the category of $k\calC$-modules with support in $\calD$ is canonically isomorphic to $k\calD$-mod, as long as $\calC$ is EI and $\calD$ is (full) convex. Let us write $k\calD$-mod$^{\circ}$ for the subcategory of $k\calD$-mod, consisting of modules whose convex supports are exactly $\calD$. Thus $k\calC$-modules can be parametrized by their convex supports. It means that $k\calC$-mod is patched up by $k\calD$-mod$^{\circ}$, with $\calD$ running over the set of all convex subcategories of $\calC$. It motivates our improved definition of the relative projectivity for category algebras.

\begin{definition} Let $\calC$ be a finite EI category and $M$ be a $k\calC$-module. Suppose $\calD$ is a subcategory of $\calC$. Then we say $M$ is \textit{projective relative to} $\calD$, or relatively $\calD$-projective, if the canonical map, still written as $\epsilon_M$,
$$
(M\downarrow^{\calC}_{\supp^cM})\downarrow^{\supp^cM}_{\calD}\uparrow^{\supp^cM}_{\calD} \to M\downarrow^{\calC}_{\supp^cM}
$$ 
is a split surjection.
\end{definition}

It is known from \cite{Xu1} that $\calD$ contains all the $M$-minimal objects. Lemma 4.6 is improved by the following statement in \cite{Xu1}. Here we offer a different proof.

\begin{proposition} Suppose $M$ is a $k\calC$-module and $M$ is relatively $\calD$-projective for a full subcategory $\calD\subset\supp^cM$. Then 
$$
(M\downarrow^{\calC}_{\supp^cM})\downarrow^{\supp^cM}_{\calD}\uparrow^{\supp^cM}_{\calD} \to M\downarrow^{\calC}_{\supp^cM}
$$
is an isomorphism. Under the circumstance $\supp^c(M\downarrow_{\calD})=\calD$.
\end{proposition}

\begin{proof} Without loss of generality, we assume $\supp^cM=\calC$. Then by assumption $M\downarrow_{\calD}\uparrow^{\calC}\to M$ is split surjective. It implies that, for each $x\in\Ob\calC$, there is a map $M(x)\to M\downarrow_{\calD}\uparrow^{\calC}(x)=\lim_{\iota/x}M\downarrow_{\calD}$. But it follows from the universal property of limits that this map has to be an isomorphism, for every $x$. Since it is straightforward to check the naturality, we obtain the claimed isomorphism of modules.

To show $\supp^c(M\downarrow_{\calD})=\calD$, we only need to prove that $M\downarrow_{\calD}$ takes non-zero values on maximal objects of $\calD$. In fact, assume $x$ is a maximal object of $\calD$ and $M\downarrow_{\calD}(x)=M(x)=0$. Then, for each $y\in\Ob\supp^cM$ that admits a morphism from $x$, we have $M(y)=0$. It implies that $x\not\in\Ob\supp^cM$, which is a contradiction.
\end{proof}

When $\calC$ is a transporter category, the following result is an immediate consequence of Theorem 4.1.

\begin{proposition} Let $\PG$ be a transporter category and $\calP'$ be a $G$-subposet of $\calP$. Assume $M\in k\PG$-mod. Then the following are equivalent
\begin{enumerate}
\item $M$ is projective relative to $\calP'\rtimes G$;
\item $M\downarrow_{\PH}$ is projective relative to $\calP'\rtimes H$ for some $H\subset G$; 
\item $M\downarrow_{\calP\rtimes 1}$ is projective relative to $\calP'\rtimes 1$;
\item $M\downarrow_{\calP\rtimes K}$ is projective relative to $\calP'\rtimes K$ for every $K\subset G$. 
\end{enumerate}
\end{proposition}

\begin{proof} Without loss of generality, assume $\supp^cM=\PG$. As $k$-vector spaces, $M\downarrow_{\calP'\rtimes H}\uparrow^{\PH}\cong M\downarrow_{\calP'\rtimes G}\uparrow^{\PG}$ by Theorem 4.1 (3), for every $H\subset G$.

Consider the natural $k\PG$-morphism $\epsilon_M$
$$
M\downarrow_{\calP'\rtimes G}\uparrow^{\PG}\to M\downarrow_{\PG}.
$$
Regarded as a $k\PH$-map, it is exactly the counit of adjunction
$$
M\downarrow_{\calP'\rtimes H}\uparrow^{\PH}=(M\downarrow_{\PH})\downarrow_{\calP'\rtimes H}\uparrow^{\PH}\to M\downarrow_{\PH}.
$$
These two maps are identical as $k$-maps. Thus one of the morphism being a $k$-isomorphism will imply the same for the other. However, such a $k$-isomorphism, if exists, is automatically a module isomorphism.
\end{proof}

For a module $M$, there usually exist proper subcategories of $\supp^cM$, making $\epsilon_M$ split surjective. These subcategories do not have to be full. We shall discuss the details in the context of transporter categories. The following characterization (slightly modified from a result in \cite{Xu1}) will be used for $\calC=\PG$ and $\calD=\QH\subset\PG$.

\begin{proposition} Let $M$ be a $k\calC$-module. Suppose $\calD\subset\supp^cM$ such that the canonical map
$$
(M\downarrow^{\calC}_{\supp^cM})\downarrow^{\supp^cM}_{\calD}\uparrow^{\supp^cM}_{\calD} \to M\downarrow^{\calC}_{\supp^cM}
$$
is surjective. Then the following are equivalent:
\begin{enumerate}
\item $(M\downarrow^{\calC}_{\supp^cM}) \bigm{|} (M\downarrow^{\calC}_{\supp^cM})\downarrow^{\supp^cM}_{\calD}\uparrow^{\supp^cM}_{\calD}$;
\item there is a $k\calD$-module $N$ such that $(M\downarrow^{\calC}_{\supp^cM}) \bigm{|} N\uparrow^{\supp^cM}_{\calD}$;
\item if $0 \to A \to B \to C \to 0$ is an exact sequence of $k\calC$-modules, with supports in $\supp^cM$, which splits upon restriction to $k\calD$-sequences, then the sequence $\Hom_{k\calC}(M, B) \to \Hom_{k\calC}(M,C) \to 0$ is exact;
\item if $0 \to A \to B \to M \to 0$ is an exact sequence of $k\calC$-modules, with supports in $\supp^cM$, which splits as an exact sequence of $k\calD$-modules, then it splits as an exact sequence of $k\calC$-modules;
\item $M$ is relatively $\calD$-projective.
\end{enumerate}
\end{proposition}

We state several examples of relative projectivity below.

\begin{proposition} Every $k\PG$-module $M$ is projective relative to $\calP^c_M\rtimes S$, where $S$ is a Sylow $p$-subgroup of $G$. 
\end{proposition}

\begin{proof} The claim is due to Boisen (for fully group-graded algebra \cite{Bo1}).
\end{proof}

The next result is actually \cite[2.3.1(2)]{Xu3}. We rewrite and include it here as a generalization to a well-known statement in group representations. Note that $k\PG$ is a Gorenstein algebra.

\begin{proposition} Let $M\in k\PG$-mod. Then it is of finite projective dimension (equivalently, of finite injective dimension) if and only if, for each $x\in\Ob(\PG)$, $M_x$ is projective relative to $\{x\}\rtimes 1$.
\end{proposition}

When $\calP$ is a point, the above statement says that $M\in kG$-mod is of finite projective dimension (equivalently, projective) if and only if it is projective relative to 1.

\begin{example} We know from \cite{Xu1} that a simple $k\PG$-module is written as $S_{x,V}$. It is determined by a simple $kG_x$-module $V=S_{x,V}(x)$, and its support is $\langle x\rangle$, consisting of the $G$-orbit of $x$. The projective cover of $S_{x,V}$ is $P_{x,V}$. We know $P_{x,V}(x)$ is the projective cover of the $kG_x$-module $V$, and $P_{x,V}=P_{x,V}(x)\uparrow^{\PG}_{\{x\}\times G_x}$.

The module $P_{x,V}$ is relatively $\{x\}\times 1$-projective. While $S_{x,V}$ is relatively $\{x\}\times H$-projective, for a $p$-subgroup $H\subset G_x$ satisfying the condition that $V$ is projective relative to $H$. Comparing with \cite{Xu1}, this makes more sense.
\end{example}

We state the following result, also for the completion of the theory.

\begin{proposition} An indecomposable $k\PG$-module $P$ is projective if and only if $P$ is projective relative to $\{x\}\times 1$ for some $x\in\Ob\PG$.
\end{proposition}

\subsection{Vertices and sources}

Using the relative projectivity, we introduce the concepts of vertices and sources. To this end, we shall establish a Mackey formula. 

\begin{example} Consider Example 4.4 again. The module $\k_x\uparrow^{\PG}_{\{x\}\times 1}$ is indecomposable. If we use the method of Dade and Boisen, then its ``vertex'' can only be of the form $\calP\rtimes H$. Thus its ``vertex'' would have to be $\PG$ (by direct computation) and the ``source'' would be itself. By contrast, it is more tempting to take $\{x\}\times 1$ (or $\{y\}\times 1$) as a vertex while $\k_x$ (or $\k_y$) as a source. To show that this kind of choices are feasible in general, we need a generalized Mackey formula for transporter category algebras.
\end{example}

Suppose $\QH$ is a transporter subcategory of $\PG$. Let $N\in k\QH$-mod. Fix an element $g\in G$. We defined the conjugate ${}^gN\in k {}^g(\QH)$-module of $N$. If $g\in N_G(\QH)$, then ${}^gN$ is a $k\QH$-module. In general if $M\in k\PG$-mod is relatively $\QH$-projective, then $M\cong {}^gM$ is also relatively ${}^g(\QH)$-projective.

From Definition 3.7, if $\PG$ is a transporter category and $\QH$ is a transporter subcategory, then $\overbrace{{}_H\calQ}\rtimes H$ becomes a weak ideal. 

\begin{lemma} If $\QH$ is weakly convex in $\PG$, then $\QH$ becomes a coideal of $\overbrace{{}_H\calQ}\rtimes H$. Under the circumstance, we observe that $N_G(\QH)\subset N_G(\overbrace{{}_H\calQ}\rtimes H)$.
\end{lemma} 

\begin{proof} The first claim is by definition, so we turn to prove the second. In fact, suppose $g\in N_G(\QH)$. Let $y\in\Ob(\overbrace{{}_H\calQ}\rtimes H)$. There is a (poset) morphism $\alpha 1: y\to x$ for some $x\in\Ob(\QH)$. Since ${}^g(\alpha 1)={}^g\alpha 1: {}^gy\to{}^gx$, by definition ${}^gy\in\Ob(\overbrace{{}_H\calQ}\rtimes H)$ because ${}^gx\in\Ob(\QH)$. It follows that ${}^g(\alpha h)\in \Mor(\overbrace{{}_H\calQ}\rtimes H)$ for every $\alpha h\in\Mor(\QH)$ and $g\in N_G(\QH)$. Thus $N_G(\QH)\subset N_G(\overbrace{{}_H\calQ}\rtimes H)$.
\end{proof}

Now we are ready to establish a Mackey type formula. At first, we provide an illuminating (more or less) example.

\begin{example} Let $\calP$ be a poset, $\calQ$ and $\calR$ be subposets. Suppose $N\in k\calQ$-mod. We compute $N\uparrow^{\calP}_{\calQ}\downarrow^{\calP}_{\calR}$. To this end, we fix an object $x\in\Ob\calR$ and analyse $N\uparrow^{\calP}_{\calQ}\downarrow^{\calP}_{\calR}(x)=\lim_{\calQ_{\le x}}N$. For instance, in either of the following two cases: $N$ is the zero functor on $\calQ_{\le x}$, or $\calQ_{\le x}=\emptyset$, the limit is zero. Meanwhile, if $x$ happens to be an object of $\calQ$ (in the intersection $\calQ\cap\calR$), the limit is just $N(x)$ because $x$ is the terminal object of $\calQ_{\le x}$.

In general $\calQ_{\le x}=\calQ\cap\calP_{\le x}$. If $\calQ_{\le x}\subset\calR_{\le x}$, we will have $\calQ_{\le x}=(\calQ\cap\calR)_{\le x}$. It has the consequence that $\lim_{\calQ_{\le x}}N\cong\lim_{(\calQ\cap\calR)_{\le x}}N\downarrow_{\calQ\cap\calR}$, or $N\uparrow^{\calP}_{\calQ}\downarrow^{\calP}_{\calR}(x)\cong N\downarrow^{\calQ}_{\calQ\cap\calR}\uparrow^{\calR}_{\calQ\cap\calR}(x)$. Thus if every $x\in\Ob\calR$ satisfies the (stronger) condition that $\calP_{\le x}=\calR_{\le x}$, we obtain
$$
N\uparrow^{\calP}_{\calQ}\downarrow^{\calP}_{\calR}\cong N\downarrow^{\calQ}_{\calQ\cap\calR}\uparrow^{\calR}_{\calQ\cap\calR}.
$$
The property of $\calR$ is equivalent to saying that $\calR$ is an ideal in $\calP$ (with no reference to modules). 

Set $\supp N=\calQ_N$. In practice, we only need to ask $(\underbrace{\calQ_N})_{\le x}\subset\calR_{\le x}$ for every $x\in\Ob\calR$, in order to get $\lim_{\calQ_{\le x}}N\cong\lim_{(\calQ\cap\calR)_{\le x}}N\downarrow_{\calQ\cap\calR}$. The reason is that by Lemma 3.21
$$
\lim_{\calQ_{\le x}}N\cong\lim_{\calQ_{\le x}\cap\tiny{\underbrace{\calQ_N}}}N
$$
because $\calQ_{\le x}\cap\underbrace{\calQ_N}$ is a coideal, containing $\supp N\downarrow_{\calQ_{\le x}}=(\calQ_N)_{\le x}$, in $\calQ_{\le x}$. Moreover
$$
\lim_{\calQ_{\le x}\cap\tiny\underbrace{\calQ_N}}N=\lim_{\calQ_{\le x}\cap(\tiny\underbrace{\calQ_N})_{\le x}}N=\lim_{(\calQ\cap\calR)_{\le x}\cap\tiny\underbrace{\calQ_N}}N
$$
since the indexing posets are the same, and
$$
\lim_{(\calQ\cap\calR)_{\le x}\cap\tiny\underbrace{\calQ_N}}N\cong\lim_{(\calQ\cap\calR)_{\le x}}N\downarrow_{\calQ\cap\calR}
$$
as $(\calQ\cap\calR)_{\le x}\cap\underbrace{\calQ_N}$ is a coideal in $(\calQ\cap\calR)_{\le x}$, containing $\supp N\downarrow_{(\calQ\cap\calR)_{\le x}}$.

Note that $\underbrace{\calQ_N}=(\underbrace{\calQ_N})^c$ cannot be replaced by either $\calQ_N$ or $\calQ^c_N$.
\end{example}

With Theorem 4.1 (3), the above example can be readily extended. Let $\calP$ be a $G$-poset, and $\calQ, \calR$ be two $G$-subposets. Given $N\in k\calQ\rtimes G$-mod such that $\supp N=\calQ_N\rtimes G$ and such that $(\underbrace{\calQ_N})_{\le x}\subset\calR_{\le x}$ for every $x\in\Ob\calR$, then
$$
N\uparrow^{\PG}_{\calQ\rtimes G}\downarrow^{\PG}_{\calR\rtimes G}\cong N\downarrow^{\calQ\rtimes G}_{(\calQ\cap\calR)\rtimes G}\uparrow^{\calR\rtimes G}_{(\calQ\cap\calR)\rtimes G}.
$$

Now we prove a generalized Mackey formula. The above special form will also be used later on.

\begin{theorem}[Mackey formula for transporter categories] Suppose that $\QH$ is a transporter subcategory of $\PG$. Let $N\in k\QH$-mod and $\supp N=\calQ_N\rtimes H$. Assume $\RK$ is a transporter subcategory of $\PG$ such that ${}^g(\underbrace{\calQ_N})_{\le x}\subset\calR_{\le x}$, for all $g\in G$ and $x\in\Ob\calR$ (which implies that $\calR\cap\underbrace{{}_G(\calQ^c_N)}$ is an ideal in $\underbrace{{}_G(\calQ^c_N)}\supset\supp^c(N\uparrow^{\PG}_{\QH})$.) Then
$$
N\uparrow^{\PG}_{\QH}\downarrow^{\PG}_{\RK}=\bigoplus_{g\in[K\backslash G/H]}[{}^g(N\downarrow^{\QH}_{(\calR^g\cap\calQ)\rtimes(K^g\cap H)})]\uparrow^{\RK}_{(\calR\cap{}^g\calQ)\rtimes(K\cap{}^gH)}.
$$
\end{theorem}

\begin{proof} Since ${}^g(\calQ_N)=\calQ_{{}^gN}$, the condition on $\RK$ implies that, for all $g\in G$ and $x\in\Ob\calR=\Ob(\RK)$, 
$$
(*) \hspace{1cm}\lim_{{}^g\calQ_{\le x}}{}^gN\cong\lim_{(\calR\cap{}^g\calQ)_{\le x}}{}^gN\downarrow_{\calR\cap{}^g\calQ}.
$$

We start with the Mackey formula for fully group-graded algebras, and then analyse various functors that are involved.
$$
\begin{array}{ll}
&\ \ \ N\uparrow^{\PG}_{\QH}\downarrow^{\PG}_{\RK}\\
&\\
&=[(N\uparrow^{\calP\rtimes H}_{\QH})\uparrow^{\PG}_{\calP\rtimes H}\downarrow^{\PG}_{\calP\rtimes K}]\downarrow^{\calP\rtimes K}_{\RK} \\
&\\
\mbox{\tiny{Thm 2.1}}&=\bigoplus_{g\in[K\backslash G/H]}[{}^g(N\uparrow^{\calP\rtimes H}_{\QH}\downarrow^{\calP\rtimes H}_{\calP\rtimes(K^g\cap H)})]\uparrow^{\calP\rtimes K}_{\calP\rtimes(K\cap{}^gH)}\downarrow^{\calP\rtimes K}_{\RK} \\
&\\
\mbox{\tiny{Cor 4.5(1)}}&\cong\bigoplus_{g\in[K\backslash G/H]}[{}^gN\uparrow^{{}^g(\calP\rtimes H)}_{{}^g(\QH)}\downarrow^{{}^g(\calP\rtimes H)}_{{}^g(\calP\rtimes(K^g\cap H))}]\downarrow^{\calP\rtimes(K\cap{}^gH)}_{\calR\rtimes(K\cap{}^gH)}\uparrow_{\calR\rtimes(K\cap{}^gH)}^{\RK}  \\
&\\
&=\bigoplus_{g\in[K\backslash G/H]}[{}^gN\uparrow^{\calP\rtimes{}^gH}_{{}^g\calQ\rtimes{}^gH}\downarrow^{\calP\rtimes{}^gH}_{\calR\rtimes(K\cap{}^gH)}]\uparrow_{\calR\rtimes(K\cap{}^gH)}^{\RK} \\
&\\
\mbox{\tiny{Cor 4.5(2)}}&\cong\bigoplus_{g\in[K\backslash G/H]}[{}^gN\downarrow^{{}^g\calQ\rtimes{}^gH}_{{}^g\calQ\rtimes(K\cap{}^gH)}\uparrow^{\calP\rtimes(K\cap{}^gH)}_{{}^g\calQ\rtimes(K\cap{}^gH)}\downarrow^{\calP\rtimes(K\cap{}^gH)}_{\calR\rtimes(K\cap{}^gH)}]\uparrow_{\calR\rtimes(K\cap{}^gH)}^{\RK} \\
&\\
(*)&\cong\bigoplus_{g\in[K\backslash G/H]}[{}^gN\downarrow^{{}^g\calQ\rtimes{}^gH}_{(\calR\cap{}^g\calQ)\rtimes(K\cap{}^gH)}\uparrow^{\calR\rtimes(K\cap{}^gH)}_{(\calR\cap{}^g\calQ)\rtimes(K\cap{}^gH)}]\uparrow_{\calR\rtimes(K\cap{}^gH)}^{\RK} \\
&\\
&=\bigoplus_{g\in[K\backslash G/H]}{}^gN\downarrow^{{}^g\calQ\rtimes{}^gH}_{(\calR\cap{}^g\calQ)\rtimes(K\cap{}^gH)}\uparrow^{\RK}_{(\calR\cap{}^g\calQ)\rtimes(K\cap{}^gH)} \\
&\\
&=\bigoplus_{g\in[K\backslash G/H]}[{}^g(N\downarrow^{\QH}_{(\calR^g\cap\calQ)\rtimes(K^g\cap H)})]\uparrow^{\RK}_{(\calR\cap{}^g\calQ)\rtimes(K\cap{}^gH)}.
\end{array}
$$
\end{proof}

The conditions in the above theorem seem to be complicated and asymmetric. However we will often find ourselves in a situation where $\calR\subset\calP$ is an ideal, and then the Mackey formula can be applied. \\

\noindent ({\bf Mackey formula}) Suppose that $\QH$ is a transporter subcategory of $\PG$. Let $N\in k\QH$-mod. Assume $\RK$ is a transporter subcategory of $\PG$ such that $\calR\subset\calP$ is an ideal. Then
$$
N\uparrow^{\PG}_{\QH}\downarrow^{\PG}_{\RK}=\bigoplus_{g\in[K\backslash G/H]}[{}^g(N\downarrow^{\QH}_{(\calR^g\cap\calQ)\rtimes(K^g\cap H)})]\uparrow^{\RK}_{(\calR\cap{}^g\calQ)\rtimes(K\cap{}^gH)}.
$$

We shall demonstrate that the Mackey formula is as powerful as we would have expected. The idea is to apply the Mackey formula to transporter subcategories of $\supp^cM=\calP^c_M\rtimes G$. Fix an indecomposable $k\PG$-module $M$. We shall show there exist minimal weakly convex transporter subcategories (contained in $\supp^cM$), relative to which $M$ is projective, and furthermore they are conjugate in $\PG$. This will be established in two steps.

\begin{theorem} Let $M$ be an indecomposable $k\PG$-module. Suppose $\supp^cM=\calP^c_M\rtimes G$. Then
\begin{enumerate}
\item there is a connected weak ideal of $\calP^c_M\rtimes G$, denoted by $\QH$, unique up to conjugacy in $\PG$, such that $M$ is relatively $\QH$-projective and such that $M$ is relatively $\RK$-projective, where $\RK$ is a weak ideal of $\calP^c_M\rtimes G$, if and only if $\RK$ contains a conjugate of $\QH$. 

\item there is an indecomposable $k\QH$-module $L$, unique up to conjugacy in $N_G(\QH)$, such that $M\downarrow^{\PG}_{\calP^c_M\rtimes G} \big{|} L\uparrow^{\calP^c_M\rtimes G}_{\QH}$, and such that its convex support is exactly $\QH$. Moreover $L\bigm{|} M\downarrow_{\QH}$.
\end{enumerate}
\end{theorem}

\begin{proof} Without loss of generality, we assume $\calP^c_M=\calP$. Suppose $\QH$ is a minimal weak ideal, with respect to inclusion, such that $M$ is relatively $\QH$-projective. Then $M$ must be relatively ${}^g(\QH)$-projective and moreover ${}^g(\QH)$ has to be minimal as well, for every $g\in G$. By choice, all these ${}^g(\QH)$ have to be connected.

Let $\RK$ be another weak ideal such that $M$ is relatively $\RK$-projective. We consider the module $M\downarrow^{\PG}_{\QH}\uparrow^{\PG}_{\QH}\downarrow^{\PG}_{\RK}\uparrow^{\PG}_{\RK}$. By assumption, $M$ is a direct summand of it. It follows from the Mackey formula that there exists some $g\in [K\backslash G/H]$ and an indecomposable $L'\in k(\calR\cap{}^g\calQ)\rtimes(K\cap{}^gH)$-mod, such that $M\bigm{|}L'\uparrow^{\PG}_{(\calR\cap{}^g\calQ)\rtimes(K\cap{}^gH)}$. By the minimality of ${}^g(\QH)$, we must have $(\calR\cap{}^g\calQ)\rtimes(K\cap{}^gH)={}^g(\QH)$, that is, ${}^g(\QH)\subset\RK$. 

Set $L={}^{g^{-1}}L'\in k\QH$-mod. Then it is indecomposable and $M\bigm{|}L\uparrow^{\calP\rtimes G}_{\QH}$. From $M\downarrow^{\PG}_{\QH}\bigm{|}L\uparrow^{\PG}_{\QH}\downarrow^{\PG}_{\QH}\cong L^n\oplus L''$ (for some integer $n\ge 1$), such that $L''$ is the direct sum of modules induced from proper subcategories of $\QH$, we deduce that $L\bigm{|}M\downarrow^{\PG}_{\QH}$. If $L'''$ is another indecomposable $k\QH$-module such that $M\bigm{|}L'''\uparrow^{\calP\rtimes G}_{\QH}$, then $L'''\bigm{|}L\uparrow^{\PG}_{\QH}\downarrow^{\PG}_{\QH}$. Applying the Mackey formula again, we find that $L\cong {}^gL'''$ for some $g\in N_G(\QH)$.

The convex support of $L$ has to be the whole $\QH$, because if $L(x)=0$ at a maximal object of $\QH$, then $L\uparrow^{\PG}_{\QH}(y)=0$ for all $y\in\Ob(\PG)$ that admits a morphism from $x$. It would imply that $M(y)=0$ for those objects $y$, and thus $\supp^cM\ne\PG$, a contradiction.
\end{proof}

Since every $k\PG$-module $M$ is relatively $\calP^c_M\rtimes S$-projective (Theorem 2.3), where $S$ is a Sylow $p$-subgroup of $G$, the weak ideal $\QH$ of $\calP^c_M\rtimes G$ in the preceding theorem must satisfy the condition that $H$ is a $p$-subgroup of $G$. We also emphasize that this $\QH$ is weakly convex in $\PG$.

Given the generalized Mackey formula, we propose an explicit algorithm for finding a $\QH$ in the preceding theorem. Suppose $M\in k\PG$-mod is indecomposable. Then according to Dade \cite{Da1} and Boisen \cite{Bo1}, there exists a $p$-subgroup $H'$, minimal up to conjugations in $G$, such that $M$ is relatively $\calP^c_M\rtimes H'$-projective. By the Mackey formula, it is easy to see that $H'$ must be conjugate to $H$ in Theorem 4.19. For simplicity, we assume $H'=H$. Suppose $N$ is an indecomposable $k\calP^c_M\rtimes H$-module (unique up to conjugations by $N_G(\calP^c_M\rtimes H)=N_G(H)$), such that $M\downarrow_{\calP_M^c\rtimes G}\bigm{|} N\uparrow^{\calP^c_M\rtimes G}$. From \cite{Xu1}, we know there exists the smallest ideal $\calD$ of $\calP^c_M\rtimes H$, such that $N\cong N\downarrow_{\calD}\uparrow^{\calP^c_M\rtimes H}$. However, by Lemma ?, $\calD$ must be of the form $\calV\rtimes H$, for some $H$-ideal $\calV$ of $\calP^c_M$. The subcategory $\calV\rtimes H$ is a weak ideal in $\calP^c_M\rtimes G$, and a weakly convex transporter subcategory in $\PG$. (Different choices of $N$ will result in different $\calV'\rtimes H$ which are conjugate to $\calV\rtimes H$, by elements of $N_G(H)$.) We shall prove that $\calV\rtimes H$ meets the requirements in Theorem 4.19.

The above $\calV\rtimes H$, and its conjugates, are actually minimal weakly convex transporter subcategories, relative to which $M$ is projective.

\begin{proposition} Let $M$ be an indecomposable $k\PG$-module. Suppose $\RK$ is a weakly convex transporter subcategory of $\calP^c_M\rtimes G$, relative to which $M$ is projective. Then a conjugate of $\calV\rtimes H$, as in the preceding paragraph, is contained in $\RK$.
\end{proposition}

\begin{proof} Without loss of generality, we assume $\supp^cM=\PG$. The module $M$ is projective relative to both $\calP\rtimes H$ and $\RK$. By the Mackey formula, there exists $g\in G$ such that $H\subset {}^gK$ and $M$ is projective relative to ${}^g\calR\rtimes H$. Let $L$ be an indecomposable $k{}^g\calR\rtimes H$-module such that $M\bigm{|} L\uparrow^{\PG}_{{}^g\calR\rtimes H}$. The module $N=L\uparrow^{\calP\rtimes H}$ is indecomposable as a $k\calP\rtimes H$-module, satisfying that $M\bigm{|}N\uparrow^{\PG}$. Since the indecomposable $k\PH$-module $N$ is projective relative to ${}^g\calR\rtimes H$, by the construction of $\calV\rtimes H$, we get $\calV\rtimes H\subset{}^g\calR\rtimes H$. It implies that ${}^{g^{-1}}(\calV\rtimes H)\subset\RK$.
\end{proof}

By the minimality, we know that $\calV\rtimes H$, constructed before Proposition 4.20, are conjugate to those $\QH$ in Theorem 4.19.

Now we are ready to introduce vertices and sources for indecomposable modules.

\begin{definition} Let $M$ be an indecomposable $k\PG$-module. Then a minimal weakly convex transporter subcategory $\calV_M\rtimes H\subset\calP^c_M\rtimes G$, relative to which $M$ is projective, is called a \textit{vertex} of $M$. An indecomposable $k\calV_M\rtimes H$-module $L$, such that $M\downarrow_{\calP^c_M\rtimes G}\bigm{|}L\uparrow^{\calP^c_M\rtimes G}_{\calV_M\rtimes H}$, is called a \textit{source} for $M$.
\end{definition}

The source $L$ for $M$ has (convex) support $\calV_M\rtimes H$.

\begin{proposition} If $N$ is an indecomposable $k\RK$-module with vertex $\QH$ and source $U$, then the (convex) support of $N$ must be $\underbrace{{}_K\calQ}\rtimes K$.
\end{proposition} 

\begin{proof} In fact, by direct calculations, $U\uparrow^{\RK}_{\QH}$, and hence $N$ can only possibly takes non-zero values on $\underbrace{{}_K\calQ}\rtimes K$ (state this fact in an earlier subsection). From $U\bigm{|} N\downarrow^{\RK}_{\QH}$, we find $N\downarrow^{\RK}_{\QH}$ is non-zero on every object of $\QH$. It implies $N$ is non-zero on every object of ${}_K\calQ\rtimes K$, and thus always non-zero on every object of $\underbrace{{}_K\calQ}\rtimes K$.
\end{proof}

The upcoming result demonstrate connections between our constructions and the classical ones.

\begin{proposition} Let $\calP$ be a connected $G$-poset. Let $M$ be an indecomposable $kG$-module with a vertex $Q$. Suppose $\kappa_M$ is the restriction of $M$ along $\PG\to G$. Then $\kappa_M$ is an indecomposable $k\PG$-module with a vertex $\calP\rtimes Q$.
\end{proposition}

\begin{proof} Since $LK_{\pi}\kappa_M=M$, $\kappa_M$ is indecomposable if and only if $M$ is. Meanwhile the convex support of $\kappa_M$ is the whole category $\PG$. Thus $\uparrow^{\PG}_{\calP\rtimes H}\cong \uparrow^G_H$ and $\downarrow^{\PG}_{\calP\rtimes H}\cong \downarrow^G_H$, for every subgroup $H\subset G$. It follows that $\calP\rtimes Q$ is a vertex of $\kappa_M$.
\end{proof}

For any finite group $G$, $\calS_p^1$ (the poset of all $p$-subgroups) is contractible (hence connected). If $G$ has a non-trivial normal $p$-subgroup, then $\calS_p$ is contractible. If $G$ is finite Chevalley group of characteristic $p$ and rank $\ge 2$, then $\calS_p$ is connected.

The vertices of an indecomposable module $M$ are conjugate by elements of $G$. While given a vertex $\calV_M\rtimes H$, the sources are unique up to conjugation by elements of $N_G(\calV_M\rtimes H)$. It is important to know that the convex supports of sources are exactly vertices (not proper subcategories). It means that our parametrization of indecomposable modules via their convex supports makes sense.

\begin{example} Based on Example 4.13, we see that the simple $k\PG$-module $S_{x,k}$ has $\{x\}\times T_x$ as a vertex, where $T_x$ is a Sylow $p$-subgroup of $G_x$. Meanwhile its projective cover $P_{x,k}$ has $\{x\}\times 1$ as a vertex. More generally, we can deduce that $P_{x,V}$ has $\{x\}\times 1$ as a vertex, and $S_{x,V}$ has $\{x\}\times H_x$ as a vertex, where $H_x$ is a vertex (in the classical sense) of the simple $kG_x$-module $V$.

Let us provide one more concrete example. If $H$ is a subgroup of $G$ and set $\calP=G/H=\{g_1H=H,\cdots,g_nH\}$, then the vertices of $\k$ are identified with the classical vertices of the $kH$-module $k$, through the category equivalence between $\calP\rtimes G$ and $H$. In fact, fixing a Sylow $p$-subgroup $S\subset G$, there is some $g_i$ such that $S^{g_i}\cap H$ becomes a Sylow $p$-subgroup of $H$. Then $\{H\}\times(S^{g_i}\cap H)$ is a vertex of $\k\in k\PG$-mod. 

Note that, under Dade's construction, a ``vertex'' of $\k\in k\PG$-mod would be $\calP\rtimes (S^{g_i}\cap H)$, which contains $\{H\}\times (S^{g_i}\cap H)$ as a proper subcategory. In fact, Dade asserted that every indecomposable $k\PG$-module is projective relative to $\calP\rtimes S^{g_i}$. The connection between his approach and ours is established as follows. The (po)set $\calP$ is a disjoint union of $S^{g_i}$-orbits, and we denote by $\Gamma=\{\calP_t\}_t$ the set of all these $S^{g_i}$-orbits. They are convex subposets of $\calP$, such that $\calP\rtimes S^{g_i}=\coprod_{\Gamma}\calP_t\rtimes S^{g_i}$ is a disjoint union of connected components, which are groupoids. Suppose $\calP_1=\calO_{S^{g_i}}(H)$. Then the skeleton of $\calP_1\rtimes S^{g_i}$ is exactly $\{H\}\times(S^{g_i}\cap H)$.
\end{example}

Next, we shall provide a generalized Green correspondence for modules. Based on our new definition of the vertex, it is necessary to note that for posets, one should expect something different from the Green correspondence for group modules. 

\begin{example} We may consider $\calP : x \to y \to z$ and the subposets $\calQ=\{x\}$ and $\calR : x\to y$. The $k\calP$-modules $k\to 0\to 0$, $k\to k\to 0$ and $k\to k\to k$ all have $\calQ$ as a vertex. In fact, they are all indecomposable $k\calP$-modules with this property. If we consider $k\calR$-modules, then there are two indecomposable modules with vertex $\calQ$. Thus there does not exist a 1-1 correspondence between the set of isomorphism classes of indecomposable $k\calR$-modules, with vertex $\calQ$, and that of isomorphism classes of indecomposable $k\calP$-modules, with vertex $\calQ$. 

However, we notice that either of the three modules $k\to 0\to 0$, $k\to k\to 0$ and $k\to k\to k$ determines the other two via the restriction or the left Kan extension. Thus if we consider the isomorphism classes of indecomposable modules with vertex $\calQ$, of ``maximal support'', then there is a 1-1 correspondence (between $k\to k\to 0$ and $k\to k\to k$). Moreover, the way they determine each other is clear.
\end{example}

We shall bear in mind that for a transporter category $\QH$ and an indecomposable $\QH$-module $M$, its support and convex support are identical. Moreover, indecomposable $k\PG$-modules are stratified by their supports. It helps us to formulate a generalized Green correspondence.

\begin{theorem}[The Green correspondence for transporter categories] Let $\PG$ be a connected transporter category. Suppose $\QH$ is a connected $p$-weakly ideal transporter subcategory of $\PG$. Let $\RK$ be a connected weakly convex transporter subcategory containing $\calQ\rtimes N_G(\QH)$. Then there is a one-to-one correspondence between the set of isomorphism classes of indecomposable $k\PG$-modules with vertex $\QH$ and convex support $\underbrace{{}_G\calQ}\rtimes G$, and the set of isomorphism classes of indecomposable $k\RK$-modules with vertex $\QH$ and convex support $\underbrace{{}_K\calQ}\rtimes K$.
\end{theorem}

\begin{proof} At first, we assume $\RK$ is a connected weak ideal of $\PG$. On the one hand, let $N$ be an indecomposable $k\RK$-module with vertex $\QH$ and convex support $\RK$. We construct an indecomposable $k\PG$-module $L$ with convex support $\PG$ and vertex $\QH$. To this end, we show $N\uparrow^{\PG}_{\RK}$ has a unique summand with vertex $\QH$, while other summands are projective relative to transporter subcategories of the form $(\RK)\cap{}^g(\QH)$ for some $g\not\in K$. Suppose $U$ is a source for $N$. Then $U\uparrow^{\RK}_{\QH}\cong N\oplus V$ for some $k\RK$-module $V$. Put $N\uparrow^{\PG}_{\RK}\downarrow^{\PG}_{\RK}\cong N\oplus N'$ and $V\uparrow^{\PG}_{\RK}\downarrow^{\PG}_{\RK}\cong V\oplus V'$. Then by the Mackey formula
$$
U\uparrow^{\PG}_{\QH}\downarrow^{\PG}_{\RK}=\bigoplus_{g\in[K\backslash G/H]}\{{}^g[U\downarrow^{\QH}_{(\calR^g\cap\calQ)\rtimes(K^g\cap H)}]\}\uparrow^{\RK}_{(\calR\cap{}^g\calQ)\rtimes(K\cap{}^gH)},
$$
which is also isomorphic to $N\oplus N'\oplus V\oplus V'$. There exists a $g\in K$ and its corresponding summand is $U\uparrow^{\RK}_{\QH}\cong N\oplus V$. The summands of $N'$ and $V'$ are all projective relative to transporter subcategories of the form $(\calR\cap{}^g\calQ)\rtimes(K\cap{}^gH)$ for $g\not\in K$. 

Next we show $N\uparrow^{\PG}_{\RK}$ has a unique summand with vertex $\QH$, and the other summands have vertices contained in some $(\QH)\cap{}^g(\QH)$ for $g\not\in K$. Let $L$ be an indecomposable summand of $N\uparrow^{\PG}_{\RK}$, which on restriction to $k\RK$ has $N$ as a summand. It must have $\QH$ as its vertex. (It is projective relative to $\QH$ because $N$ is. However, its vertex cannot be conjugate to a proper weakly convex transporter subcategory of $\QH$, since otherwise $N$ would be projective relative to this weakly convex transporter subcategory of $\QH$.) We want to prove that $L$ is the unique summand having $\QH$ as a vertex. Let $L'$ be another summand of $N\uparrow^{\PG}_{\RK}$. Then $L'\downarrow^{\PG}_{\RK}$ must be a direct summand of $N'$, which is projective relative to transporter subcategories of the form $(\calR\cap{}^g\calQ)\rtimes(K\cap{}^gH)$ for $g\not\in K$. Since $L'$ is a summand of $U\uparrow^{\PG}_{\QH}$, $L'$ is projective relative to $\QH$. Thus $L'$ has a vertex $\calQ'\rtimes H'$ contained in $\QH$. Since $\calQ'\rtimes H'\subset\RK$, $L'\downarrow^{\PG}_{\RK}$ has a summand which on restriction to $\calQ'\rtimes H'$ has a summand as a source for $L'$. It follows that there is a ${}^t(\calQ'\rtimes H')$, some $t\in K$, contained in one of the transporter subcategories, say $(\calR\cap{}^s\calQ)\rtimes(K\cap{}^sH)$ for $s\not\in K$. Thus $\calQ'\rtimes H'\subset{}^g(\QH)$ for $g=t^{-1}s\not\in K$ (which is not in $N_G(\QH)\subset K$). It means that $\calQ'\rtimes H'\subset(\QH)\cap{}^g(\QH)\varsubsetneq\QH$.

By Proposition 4.22, $L$ is supported on $\underbrace{{}\calQ_G}\rtimes G$.

On the other hand, assume $M$ is an indecomposable $k\PG$-module with vertex $\QH$, support $\underbrace{{}_G\calQ}\rtimes G$ and a source $S\in k\QH$-mod. We construct an indecomposable $k\RK$-module $W$ with vertex $\QH$ and support $\underbrace{{}_K\calQ}\rtimes K$. Since $M$ is a summand of $S\uparrow^{\RK}_{\QH}\uparrow^{\PG}_{\RK}$, there is a summand $W\bigm{|}S\uparrow^{\RK}_{\QH}$ such that $M\bigm{|}W\uparrow^{\PG}_{\RK}$. The $k\RK$-module $W$ must have vertex $\QH$, because it it projective relative to it, and moreover cannot be projective relative to a smaller transporter subcategory. The module $M\downarrow^{\PG}_{\RK}$ is a summand of $W\uparrow^{\PG}_{\RK}\downarrow^{\PG}_{\RK}$. By our preceding arguments, it must have only one summand with vertex $\QH$, that is, $W$.  This indecomposable module $W$, by Lemma 4.25, has $\underbrace{{}_K\calQ}\rtimes K$ as its support.

We readily verify that our previous constructions produce a 1-1 correspondence.

In the general situation, given a weakly convex $\RK$, we may produce a weak ideal $\overbrace{{}_K\calR}\rtimes K$. Since it contains $\calQ\rtimes N_G(\QH)$, by the above discussions, there is a one-to-one correspondence between the isomorphism classes of indecomposable $k\PG$-modules with vertex $\QH$ and support $\underbrace{{}_G\calQ}\rtimes G$, and the isomorphism classes of indecomposable $k\overbrace{{}_K\calR}\rtimes K$-modules with vertex $\QH$ and support $\underbrace{{}_K\calQ}\rtimes K$. However, the Kan extension $U\uparrow^{\tiny{\overbrace{{}_K \calR}}\rtimes K}_{\QH}$ of any $k\QH$-module $U$ has its support contained in $\RK$, because $\calQ$ is inside a $K$-subposet $\calR$. It implies that our correspondence is truly a correspondence between the isomorphism classes of indecomposable $k\PG$-modules with vertex $\QH$ and support $\underbrace{{}_G\calQ}\rtimes G$, and  the isomorphism classes of indecomposable $k\RK$-modules with vertex $\QH$ and support $\underbrace{{}_K\calQ}\rtimes K$.
\end{proof}

There are two special cases that we may apply the Green correspondence. One is $\calP\rtimes H \subset\PG$, for suitable $H$, and the other is $\QH\subset\calP\rtimes H$, for $\calQ\subset\calP$. These are already given by \cite{Da1} and \cite{Xu1}, respectively. In light of Proposition 4.20, we may compose maps in these special Green correspondences and obtain a result similar to the above. However, if we were to use these procedures directly, we would have to ask $N_G(H)\subset K$, instead of the slightly weaker condition $N_G(\QH)\subset K$.

The following result establishes a clear connection between category representations and group representations.

\begin{corollary} Let $\PG$ be a transporter category and $x$ be an object. Let $H\subset G_x$ be a $p$-subgroup. Suppose $K\supset N_{G_x}(H)$. Then there is a one-to-one correspondence between the set of isomorphism classes of indecomposable $k\PG$-modules with vertex $\{x\}\times H$ and support ${}_G{x}\rtimes G$, and the set of isomorphism classes of indecomposable $k\{x\}\times K$-modules with vertex $\{x\}\times H$ (and support $\{x\}\times K$).
\end{corollary}

\section{Block Theory}

Boisen \cite{Bo1} studied the block theory of fully group-graded algebras. In particular, based on Dade's work, he defined ``defect groups'' of blocks, and established a generalized Brauer's First Main Theorem. These constructions and results certainly are valid for transporter category algebras. However, Boisen's ``defects'' are not truly subgroups, and they are too big in the case of transporter category algebras. We shall improve a few of the existing results, and then propose some entirely new constructions and theorems. Especially, for a transporter category algebra, we can talk about defect transporter categories of its blocks. Since there is a trivial representation $\k$, we also have a notion of the \textit{principal block} of a transporter category algebra.

\subsection{Defect transporter categories} As we mentioned in Section 2, Boisen used a subalgebra, $\Delta(A)\subset A^e$, to introduce the ``defect groups'' of a block of a fully group-graded algebra $A$. His main observation is the module isomorphism $A_1\uparrow^{A^e}_{\Delta(A)}\cong A$. When $A=k\PG$, $A_1\cong k\calP$ and we have seen that $\Delta(k\PG)\cong k\calP^e\rtimes\delta(G)$. Boisen's ``defect groups'' would be minimal $p$-subgroups $D\subset G$ such that $k\PG$ is projective relative to $k\calP^e\rtimes\delta(D)$. We shall observe that Boisen's isomorphism $k\calP\uparrow^{\calP^e\rtimes  G^e}_{\calP^e\rtimes\delta(G)}\cong k\PG$ comes from another isomorphism $\k\uparrow^{\calP^e\rtimes G^e}_{F(\calP)\rtimes\delta(G)}\cong k\calP$. This prompts us to introduce the defect transporter subcategories of a block of $k\PG$ as subcategories of $F(\calP)\rtimes\delta(G)$. We shall explain the ideas now.

In Section 3, we constructed the following transporter categories and faithful functors
$$
F(\calP)\rtimes\delta(G) \longrightarrow \calP^e\rtimes\delta(G) \longrightarrow \calP^e\rtimes G^e\cong (\PG)^e.
$$
Passing to module categories, the above functors give rise to restrictions
$$
kF(\calP)\rtimes\delta(G)\mbox{-mod} {\buildrel{\downarrow^{\calP^e\rtimes\delta(G)}_{F(\calP)\rtimes\delta(G)}}\over{\longleftarrow}} k\calP^e\rtimes\delta(G)\mbox{-mod} {\buildrel{\downarrow^{\calP^e\rtimes G^e}_{\calP^e\rtimes\delta(G)}}\over{\longleftarrow}} k\calP^e\rtimes G^e\mbox{-mod},
$$
and their left adjoints 
$$
kF(\calP)\rtimes\delta(G)\mbox{-mod} {\buildrel{\uparrow^{\calP^e\rtimes\delta(G)}_{F(\calP)\rtimes\delta(G)}}\over{\longrightarrow}} k\calP^e\rtimes\delta(G)\mbox{-mod} {\buildrel{\uparrow^{\calP^e\rtimes G^e}_{\calP^e\rtimes\delta(G)}}\over{\longrightarrow}} k\calP^e\rtimes G^e\mbox{-mod},
$$
Subsequently, we would like to demonstrate that $F(\calP)\rtimes\delta(G)\subset\calP^e\rtimes G^e$ plays the role of the diagonal subgroup in the block theory of group algebras. Since $F(\calP)\rtimes\delta(G)$ is a coideal, hence convex, in $\calP^e\rtimes\delta(G)$, every $kF(\calP)\rtimes\delta(G)$-module is naturally a $k\calP^e\rtimes\delta(G)$-module. It means that 
$$
\uparrow^{\calP^e\rtimes\delta(G)}_{F(\calP)\rtimes\delta(G)} : kF(\calP)\rtimes\delta(G)\mbox{-mod} \to k\calP^e\rtimes\delta(G)\mbox{-mod}
$$
is just an embedding.

In what follows, we shall regard $F(\calP)\rtimes\delta(G)$ as a weakly convex transporter subcategory of $(\PG)^e\cong\calP^e\rtimes G^e$. Consider the $(k\PG)^e$-module $k\PG$. Then $\supp^c(k\PG)=\calP^e_{k\PG}\rtimes G^e$ consists of objects $\{(y,x^{op})\bigm{|}\Hom_{\PG}(x,y)\ne\emptyset\}$. Note that $F(\calP)$, identified with a subposet of $\calP^e$, is convex but not ideal in $\calP^e_{k\PG}$.

\begin{lemma} Let $\k\in kF(\calP)\rtimes\delta(G)$-mod. Then $\k\uparrow^{\calP^e\rtimes\delta(G)}_{F(\calP)\rtimes\delta(G)}\cong k\calP$ as $k\calP^e\rtimes\delta(G)$-modules. Consequently $\k\uparrow^{(\PG)^e}_{F(\calP)\rtimes\delta(G)}\cong k\PG$ as $(k\PG)^e$-modules.
\end{lemma}

\begin{proof} The first isomorphism follows from $\k\uparrow^{\calP^e\rtimes\delta(G)}_{F(\calP)\rtimes\delta(G)}\cong\k\uparrow^{\calP^e}_{F(\calP)}\cong k\calP$, see \cite{Xu2} for the latter isomorphism. 

Along with Boisen's result, our first statement gives rise to the second.
\end{proof}

\begin{remark} When talking about blocks of a category algebra $k\calC$, we usually assume $\calC$ to be connected. If not, then $k\calC$ becomes a direct product $\prod_i k\calC_i$, where $\calC_i$ rans over the set of connected components of $\calC$. To study blocks of $k\calC$, it suffices to examine each $k\calC_i$. There is one more advantage to study connected categories. If $\calC$ is (finite) connected, then  the blocks of $k\calC$, as $k\calC^e$-modules, are non-isomorphic.

Now let $\PG$ be connected. (We shall emphasize that the connectedness of $\PG$ does not imply the connectedness of $\calP$.) Then the $k\calP^e\rtimes\delta(G)$-module $k\calP$ is indecomposable. The support of $k\PG$ is $\calC$ whose objects are identified with those of the image of $\Ob F(\calP\rtimes G)$ in $\Ob(\calP^e\rtimes G^e)$. Thus each block of $k\PG$ has (convex) support contained in $\calC$.
\end{remark}

Suppose $\PG$ is a connected transporter category. Let $B$ be a block of $k\PG$. Then $B\bigm{|}k\PG$ as $(k\PG)^e$-modules. Since $\k\uparrow^{(\PG)^e}_{F(\calP)\rtimes\delta(G)}\cong k\PG$, a vertex of $B$ lies in $F(\calP)\rtimes\delta(G)$.

\begin{definition} Suppose $\PG$ is a connected transporter category. Let $B$ be a block of $k\PG$. Regarded as a $(k\PG)^e$-module, a vertex $\calV\rtimes\delta(D)$ of $B$ that is contained in $F(\calP)\rtimes\delta(G)$ is called a \textit{defect transporter category}, or simply a defect, of $B$.
\end{definition}

The block theory of transporter category algebras will be discussed in a parallel paper. It is interesting, because there are enough blocks. The simplest example will be that $k(G/H)\rtimes G\simeq kH$ for a subgroup $H$. Moreover, Peter Webb constructed examples where the blocks of a group algebra biject with those of a certain transporter category algebra (which is not a group algebra). 

To finish off, we use a couple of examples to illustrate some features of the theory. Unlike group representations, the defects of the block $B$ do not have to be conjugate by elements of $\delta(G)$.

\begin{example} Let $\calP_n=x_1{\buildrel{\alpha_1}\over{\to}x_2\to\cdots\to}x_{n-1}{\buildrel{\alpha_{n-1}}\over{\to}}x_n$. Then there is only one block. Its defect is the vertex of $\k\in kF(\calP_n)$-mod, which is the following subposet $\calV\subset F(\calP_n)$.
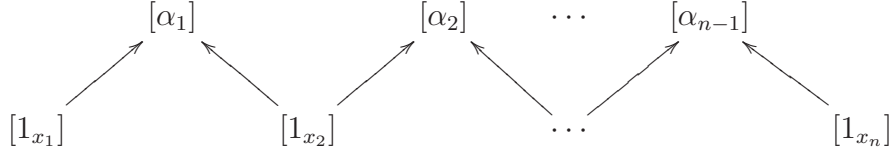
$$
\xymatrix{&[\alpha_1]&&[\alpha_2]&\cdots&[\alpha_{n-1}]&\\
[1_{x_1}]\ar[ur]&&[1_{x_2}]\ar[ur]\ar[ul]&&\cdots\ar[ur]\ar[ul] && [1_{x_n}]\ar[ul]}
$$
\end{example}

Let $M$ be an indecomposable $k\PG$-module that lies in a block $B$. It is not necessarily true that a vertex $\calV_M$ of $M$ satisfies the condition that $F(\calV_M)\subset\calD_B$, where $\calD_B$ is a defect of $B$.

\begin{example} Let $\calP$ be the following poset
$$
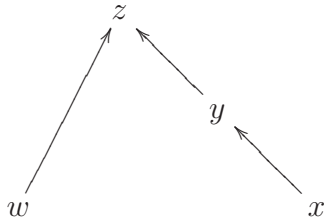
\xymatrix{&z&&\\
&&y\ar[ul]&\\
w\ar[uur]&&&x\ar[ul]}
$$
It is connected so there is only one block $B_0$ (called the principal block). Thus every module lies in the block $B_0$. The vertex of $\k\in k\calP$-mod is the whole poset $\calP$. However, when we examine the vertex of $k\calP$ as a $k\calP^e$-module. Then we find that the defect of $B_0=k\calP$ is a proper subposet of $F(\calP)$.
\end{example} 

\subsection{Brauer correspondent} Suppose $A$ is a fully group-graded algebra. Boisen \cite{Bo1} introduced a Brauer correspondence between blocks of $A$ and of $A_H$ for suitable subgroup $H\subset G$. Assume $A=\SG$ is a skew group algebra. Let $b$ be a block of $A_H$ and $B$ be a block of $A$. Then $B$ is said to correspond to $b$ if $B$ is the unique block such that $b\bigm{|} B\downarrow^{A^e}_{A_H^e}$. We shall be interested in the case when $A=k\PG$, and improve Boisen's construction.

\begin{definition} Suppose $\PG$ is a connected transporter category and $\QH$ is a connected transporter subcategory. Let $B$ be a block of $k\PG$ and $b$ be a block of $k\QH$. We say $B$ corresponds to $b$, written as $B=b^{\PG}$, if $B$ is the unique block of $k\PG$ such that $b\bigm{|} B\downarrow^{(\PG)^e}_{(\QH)^e}$.

If $b$ is a block of $k\QH$ which has a block of $k\PG$ corresponds to it, then we say $b^{\PG}$ is defined.
\end{definition}

Suppose $\PG$ is a connected transporter subcategory and $\QH$ is a connected weakly convex transporter subcategory. Let $b$ be a block of $k\QH$. If $\QH$ is contained in some transporter subcategory $\RK$ while $b^{\RK}$, $(b^{\RK})^{\PG}$ and $b^{\PG}$ are defined, then we have an equality $b^{\PG}=(b^{\RK})^{\PG}$.

Based on his definition, Boisen \cite{Bo1} continued to establish a generalized Brauer's First Main Theorem for fully group-graded algebras. We state relevant consequences for a transporter category algebra $k\PG$. 

\begin{enumerate}
\item If $C_G(D)\subset H$, then $b^{\PG}$ is defined for any block $b$ of $k\PH$ which is projective relative to $\calP^e\rtimes D$ with $D$ a minimal $p$-subgroup with respect to this property. 

\item (Brauer correspondence for skew group algebras) Suppose $H$ is a subgroup and $D$ is a $p$-subgroup of $G$ such that $N_G(D)\subset H$. Then there is a one-to-one correspondence between the blocks of $k\PH$ that are relatively $k\calP^e\rtimes D$-projective for minimal $D$, and the blocks of $k\PG$ that are relatively $k\calP^e\rtimes D$-projective for minimal $D$. 

\item With Lemma 5.1, (2) can be restated as follows. Under the same assumptions, there is a one-to-one correspondence between the blocks of $k\PH$ that are relatively $F(\calP)\rtimes D$-projective for minimal $D$, and the blocks of $k\PG$ that are relatively $F(\calP)\rtimes D$-projective for minimal $D$.
\end{enumerate}

Boisen's Brauer correspondence has a counterpart, when group is fixed and subposets vary.

\begin{lemma} Suppose $\PG$ is a transporter category and $\QH$ is a connected weakly convex transporter subcategory. Let $b$ be a block of $k\QH$, with defect $\VD$. If $N_G(D)\subset H$, then $b^{\PH}$ is defined. It is the Green correspondent of $b$, and has $\VD$ as its defect.
\end{lemma}

\begin{proof} Let $b_1$ be a block of $k\PH$ with defect $\VD$. Then $b_1\downarrow_{(\QH)^e}$ has a unique direct summand $b$, which is the Green correspondent of $b_1$. Since $k\PH\downarrow^{(\PH)^e}_{(\QH)^e}=k\QH$, it is clear that $b$ is a block of $k\QH$ with defect $\VD$.

Conversely if $b$ is a block of $k\QH$ with defect $\VD$, then there exists a block $b_2$ of $k\PH$ so that $b\bigm{|}b_2\downarrow_{(\QH)^e}$. Since the blocks of $k\PH$ (and of $k\QH$) are non-isomorphic, $b_2$ is unique.

Thus we have a one-to-one correspondence between blocks of $k\QH$ with defect $\VD$, and blocks of $k\PH$ with the same defect, via $b\mapsto b^{\PH}$.
\end{proof}

Along with the Brauer correspondence for group-graded algebras by Boisen, we will obtain an improved correspondence between blocks of transporter category algebras.

\begin{theorem}[Brauer's First Main Theorem for transporter categories] Suppose $\PG$ is a (connected) transporter category and $\QH$ is a (connected) weakly convex transporter subcategory. Let $\VD$ be a connected weakly convex $p$-transporter subcategory of $F(\calP)\rtimes G$, such that $\calV\subset F(\calQ)$ and $N_G(D)\subset H$. Then there is a one-to-one correspondence between the blocks of $k\QH$ with defect $\VD$, and the blocks of $k\PG$ with defect $\VD$, given by letting $b$ correspond to $b^{\PG}$.
\end{theorem}

\begin{proof} Keep the following diagram of categories in mind, where each arrow represents an inclusion.
$$
\xymatrix{&(\QH)^e \ar[r] & (\PH)^e \ar[r] & (\PG)^e\\
\VD \ar[r] & F(\calQ)\rtimes D \ar[r] \ar[u] & F(\calP)\rtimes D \ar[u] & }
$$
Let $B$ be a block of $k\PG$ with defect $\VD$. Since it is projective relative to $\calP^e\rtimes D$, there is a unique block $b_1$ of $k\PH$, which is projective relative to $\calP^e\rtimes D$, such that $b_1\bigm{|} B\downarrow_{(\PH)^e}$. It means $b_1^{\PG}=B$. We shall show $b_1$ has $\VD$ as a defect. Then Lemma 5.7 identifies a unique block $b$ of $k\QH$ with defect $\VD$ such that $b^{\PH}=b_1$. Since $B=b_1^{\PG}=(b^{\PH})^{\PG}=b^{\PG}$, we are done.

Let $S\in k\VD$-mod be a source for $B$. Then 
$$
b_1\downarrow_{\calP^e\rtimes D}\bigm{|}S\uparrow^{\calP^e\rtimes G^e}\downarrow_{\calP^e\rtimes D}
$$
Now by the Mackey formula, the right hand side is 
$$
\bigoplus_{(g_1,g_2)\in[D\backslash G^e/D]}[{}^{(g_1,g_2)}(S\downarrow^{\calP^e\rtimes G^e}_{[\calP^e\rtimes D]^{(g_1,g_2)}\cap(\VD)})]\uparrow^{\calP^e\rtimes G^e}_{[\calP^e\rtimes D]\cap{}^{(g_1,g_2)}(\VD)}.
$$
By assumption $D$ is minimal with respect to the property that $b_1$ is relatively $\calP^e\rtimes D$-projective. Since $(\calP^e\rtimes D)\cap{}^{(g_1,g_2)}(\VD)={}^{(g_1,g_2)}\calV\rtimes(D\cap{}^{(g_1,g_2)}D)$, it implies that, for some $(g_1,g_2)$, $D\cap{}^{(g_1,g_2)}D=D$. (While the other intersection groups are strictly smaller than $D$.) It forces $D={}^{(g_1,g_2)}D$ and ${}^{(g_1,g_2)}\calV\rtimes(D\cap{}^{(g_1,g_2)}D)={}^{(g_1,g_2)}\calV\rtimes{}^{(g_1,g_2)}D$ becomes a conjugate of $\VD$. Thus $b_1$ must have $\VD$ as a defect.
\end{proof}

Under the assumption of the above theorem, $b$ is called the \textit{Brauer correspondent} of $b^{\PG}$. Note that to guarantee the existence of $b^{\PG}$, we only need to require $C_G(D)\subset H$, by Boisen \cite{Bo1} and Lemma 5.7. In a separate paper, we shall develop a ring-theoretic approach, and try to generalize some other results, including the second and third main theorems of Brauer.

\end{document}